\documentclass{amsart}
\usepackage{amsmath,url,amssymb,amsthm,amsfonts,todonotes,bbm, mathrsfs}
\usepackage[all]{xy}
\usepackage{eucal}
\usepackage{fullpage}
\usepackage[T1]{fontenc}
\usepackage{textcomp}
\usepackage{lmodern}
\usepackage[english]{babel}
\usepackage{extarrows}
\usetikzlibrary{patterns,decorations.pathreplacing}

\usepackage[pagebackref]{hyperref}
  
\hypersetup{%
  bookmarksnumbered=true,
  colorlinks=true,%
  linkcolor=blue,%
  citecolor=blue,%
  filecolor=blue,%
  menucolor=blue,%
  urlcolor=blue,%
  pdfnewwindow=true,%
  pdfstartview=FitBH}

\newcommand{\R}{\mathbb{R}}
\newcommand{\C}{\mathbb{C}}
\newcommand{\Z}{\mathbb{Z}}

\newcommand{\N}{\mathbb{N}}

\newcommand{\E}{\mathcal{E}}
\newcommand{\I}{\mathcal{I}}
\newcommand{\J}{\mathcal{J}}

\DeclareMathOperator{\colim}{\mathrm{colim}}

\DeclareMathOperator{\hocolim}{\mathrm{hocolim}}
\DeclareMathOperator{\holim}{\mathrm{holim}}
\DeclareMathOperator{\hofibre}{\mathrm{hofibre}}

\DeclareMathOperator{\homog}{--homog--}

\DeclareMathOperator{\poly}{--poly--}

\newcommand{\s}{\mathsf{Sp}}
\DeclareMathOperator{\T}{\mathsf{Top_\ast}}

\DeclareMathOperator{\Hom}{\mathsf{Hom}}

\newcommand{\cofrep}{Q}
\newcommand{\Ev}{\mathrm{Ev}}
\newcommand{\fibrep}{R}
\DeclareMathOperator{\id}{\mathrm{id}}
\DeclareMathOperator{\ind}{\mathrm{ind}}
\DeclareMathOperator{\res}{\mathrm{res}}

\newcommand{\U}{\mathsf{U}}
\newcommand{\BU}{\mathsf{BU}}

\theoremstyle{definition}
\newtheorem{thm}{Theorem}[section]

\newtheorem{prop}[thm]{Proposition}

\newtheorem{lem}[thm]{Lemma}
\newtheorem{cor}[thm]{Corollary}
\newtheorem{example}[thm]{Example}
\newtheorem{definition}[thm]{Definition}
\newtheorem{rem}[thm]{Remark}

\newtheorem{ex}[thm]{Example}
\newtheorem{alphtheorem}{Theorem}[section]

\newtheorem{alphprop}[alphtheorem]{Proposition}

  \newcommand{\adjunction}[4]{
\xymatrix{
#1:#2 \ar@<0.7ex>[r] &
\ar@<0.7ex>[l] #3:#4
}}

\begin{document}

\title{Unitary calculus: model categories and convergence}
\author{Niall Taggart}
\address{Mathematical Institute, Utrecht University}
\email{n.c.taggart@uu.nl}
\date{}
\begin{abstract}
We construct the unitary analogue of orthogonal calculus developed by Weiss, utilising model categories to give a clear description of the intricacies in the equivariance and homotopy theory involved. The subtle differences between real and complex geometry lead to subtle differences between orthogonal and unitary calculus. To address these differences we construct unitary spectra - a variation of orthogonal spectra - as a model for the stable homotopy category. We show through a zig-zag of Quillen equivalences that unitary spectra with an action of the $n$-th unitary group models the homogeneous part of unitary calculus. We address the issue of convergence of the Taylor tower by introducing weakly polynomial functors, which are similar to weakly analytic functors of Goodwillie but more computationally tractable. 
\end{abstract}
\maketitle 

\setcounter{tocdepth}{1}
{\hypersetup{linkcolor=black} \tableofcontents}

\section{Introduction}

Functor calculus was originally developed by Goodwillie \cite{Go90, Go91, Go03} to systematically study the algebraic $K$-theory of spaces. The theory developed from the study of homotopy preserving endofunctors on $\T$ to a more varied settings, such as, functors $\T \longrightarrow \s$, and $\s \longrightarrow \s$. Biedermann, Chorny and R\"{o}ndigs \cite{BCR07} and Biedermann and R\"{o}ndigs \cite{BR14} provided a model category framework for Goodwillie calculus, with  Kuhn \cite{Kuh07} developing the theory for abstract model categories. The theory has been extended to the study of functors on suitable $(\infty,1)$-categories by Lurie \cite{Lur18}.  

As these developments of Goodwillie calculus were taking place, the general theory of a ``calculus of functors'' was developing for functors with more structure. One example of this is the orthogonal calculus of Weiss \cite{We95}. This calculus gives a framework for the systematic study of functors indexed on real inner product spaces. Key examples include $\mathsf{BO}(-): V \longmapsto \mathsf{BO}(V)$ and $\mathsf{BDiff^b}(M \times -) : V \longmapsto \mathsf{BDiff^b}(M \times V)$ where $\mathsf{BDiff^b}(M \times V)$ is the classifying space of the group of bounded diffeomorphisms from $M \times V$ to itself, for $M$ a fixed manifold.  Other examples of functor calculus include the additive functor calculus developed by Johnson and McCarthy \cite{BM99, BM03a, BM03b} and the manifold calculus of Goodwillie and Weiss \cite{We99, GW99, dBW13}.

Unitary calculus is the extension of orthogonal calculus to the study of functors from the category of complex inner product spaces to topological spaces. The orthogonal calculus relies heavily on the real geometry of the vector spaces, and hence the unitary calculus relies on the complex geometry involved. The subtle differences in the geometry lead to subtle differences in the calculi.

The theory of orthogonal calculus has been developed extensively in the literature, for example in \cite{We95, MW09, BO13} and \cite{Ba17}. Contrastingly, the unitary version does not have solid foundations in the literature despite being known to the experts. For example, in \cite[Example 10.3]{We95} Weiss calculates the first derivative of $\mathsf{BU}(-): V \longmapsto \mathsf{BU}(V)$. These calculations have been taken further in \cite{Ar02}, where Arone studies the calculus of the functor $\mathsf{BU}(-)$ in great detail, giving a closed formula for the derivatives \cite[Theorems 2 and 3]{Ar02}, and calculates homology approximations of the layers of the associated Taylor tower \cite[\S4]{Ar02}. Other examples where unitary calculus has been employed in the literature include \cite{Ar98, Ar01, ADL08} and \cite{BE16}. 

Since orthogonal calculus is built from real vector spaces, and unitary calculus is built from complex vector spaces, there should exist interesting comparisons between the calculi similar to the comparisons between $K$-theory, $\mathsf{KU}$, and real $K$-theory, $\mathsf{KO}$ induced by the complexification - realification adjunction between real and complex inner product spaces. For these comparisons to be possible, a firm grasp of the unitary calculus is essential.

Moreover, the complex vector spaces considered in unitary calculus come with a complex conjugation, which induces a $C_2\T$-enrichment on the category of input functors. This leads to the notion of ``unitary calculus with reality'', a version of unitary calculus which takes into account this $C_2$-action. Furthermore, it should then be possible to compare all three calculi allowing for the movement of computations between the calculi. The comparisons should be similar to those between $K$-theory, real $K$-theory and the $K$-theory with reality of Atiyah \cite{At66}. It is precisely these questions which the author will address in future work, utilising the solid foundations laid out in this paper.

\section{Main Results and summary of Unitary Calculus}
We give the main results of the paper and a summary of the calculus for easy reference.

\subsection*{The machinery} 
The orthogonal and unitary calculi are similar in that given a functor $F$ of the appropriate type, the calculus assigns a sequence of functors $\{T_nF\}_{n \in \N}$ to $F$. These functors are ``polynomial'' in the sense that they assemble into a Taylor tower, the layers (differences between successive polynomial approximations) of which are determined by spectra with an action of an appropriate group; an orthogonal group for orthogonal calculus, and a unitary group for unitary calculus. For the unitary calculus case, we have the following result, which is Theorem \ref{characterisation of homogeneous functors}. Note that a functor $F$ is $n$-homogeneous if it is both $n$-polynomial and $T_{n-1}F$ vanishes, see Definition \ref{definition: homogeneous functor}. 

\begin{alphtheorem}\label{thm 1}
Let $F \in \E_0$ be $n$-homogeneous for some $n >0$. Then $F$ is levelwise weakly equivalent to the functor defined as 
\begin{equation*}\label{char of homog functors}
U \longmapsto \Omega^\infty [(S^{nU} \wedge \Psi_F^n)_{h\U(n)}]
\end{equation*}
where $\Psi_F^n$ is a spectrum with an action of $\U(n)$ formed by the $n$-th derivative of $F$, and $S^{nU}$ is the one-point compactification of $\C^n \otimes U$ with the induced $\U(n)$-action via the regular representation action on $\C^n$. 
\end{alphtheorem}

For the most part, these functor calculi sit in quite strong analogy with Taylor's Theorem from differential calculus. We construct an $n$-th polynomial approximation functor $T_n$ in Section \ref{section: polynomial functors}, which sits an analogy with the Taylor polynomial $p_n(x)$ from differential calculus. Moreover the $n$-th polynomial approximation of an $n$-polynomial functor recovers the original functor, as is the case for differential calculus, see Proposition \ref{equiv to poly approx}. 

Given a functor $F$, the first derivative $F^{(1)}$ (see Definition \ref{def: derivative}), has structure maps
\[
S^2 \wedge F^{(1)}(U) \longrightarrow F^{(1)}(U \oplus V),
\]
and the second derivative has structure maps 
\[
S^4 \wedge F^{(2)}(U) \longrightarrow F^{(2)}(U \oplus V).
\]

In general, the $n$-th derivative has structure maps,
\[
S^{2n}\wedge F^{(n)}(U) \longrightarrow F^{(n)}(U \oplus V).
\]

In Proposition \ref{derivatives stable} we show that adjoint structure maps give a method of calculating the $n$-th derivative from the $(n-1)$-st derivative. 

\begin{alphprop}
Let $n$ be a non-negative integer. There is a homotopy fibre sequence 
\[
F^{(n)}(U) \longrightarrow F^{(n-1)}(U) \longrightarrow \Omega^{2(n-1)}F^{(n-1)}(U \oplus \C),
\]
for all $U \in \J_n$. 
\end{alphprop}

We further show in Proposition \ref{prop: hofib of n-poly map} that the derivative of a functor is a measure of how far a functor is from being polynomial. 

\begin{alphprop}
Let $n$ be a non-negative integer and let $F \in \E_0$. There is a homotopy fibre sequence 
\[
F^{(n+1)}(U) \longrightarrow F(U) \longrightarrow  \tau_n F(U),
\]
for all $U \in \J_0$. 
\end{alphprop}

There is a map from the $n$-th polynomial approximation to the $(n-1)$-st polynomial approximation, the homotopy fibre of which is $n$-homogeneous, hence there is a homotopy fibre sequence 
\[
\Omega^\infty[(S^{nU} \wedge \Psi_F^n)_{h\U(n)}] \longrightarrow T_nF(U) \longrightarrow T_{n-1}F(U).
\]

These homotopy fibre sequences assemble into a Taylor tower approximating the functor $F$,
\[
\xymatrix@C+2cm{
   & \vdots \ar[d]^{r_{n+1}} &\\
   & T_n F(U) \ar[d]^{r_n} &  \ar[l] \Omega^\infty[(S^{nU} \wedge \Psi_F^n)_{h\U(n)}] \\
   & T_{n-1}F(U) \ar[d]^{r_{n-1}}  &  \ar[l] \Omega^\infty[(S^{(n-1)U} \wedge \Psi_F^{n-1})_{h\U(n-1)}]\\
   & \vdots  \ar[d]^{r_2} &  \\ 
   &T_1 F(U) \ar[d]^{r_{1}}  &  \ar[l] \Omega^\infty[(S^U \wedge \Psi_F^1)_{h\U(1)}]\\
F(U) \ar@/^2pc/[uuuur]  \ar@/^1pc/[uuur]   \ar@/^/[ur] \ar[r] & F(\C^\infty).  &\\
}
\]

\subsection*{Model categories for unitary calculus}
In \cite{BO13} Barnes and Oman rewrote the homotopy theory originally developed by Weiss in \cite{We95} into the language of model categories. This is advantageous as it removed the need for ``up to homotopy'' statements and provided a clearer understanding of the equivariance in the picture. In this paper we construct the theory of unitary calculus in terms of model categories. These model categories will enable clear comparisons between the versions of the calculi. 

We start with a construction of an $n$-polynomial model structure, which contains the homotopy theory of $n$-polynomial functors. In particular weak equivalences in this model structure are detected by the $n$-th polynomial approximation functor, and moreover the $n$-th polynomial approximation functor is a model for a fibrant replacement functor in this model structure. 

In particular we can construct a model structure, the $n$-homogeneous model structure, which contains the homotopy theory of $n$-homogeneous functors. Here, the cofibrant-fibrant objects are precisely the $n$-homogeneous functors and weak equivalences are detected by the $n$-th layer of their respective towers. 

Using the characterisation of the $n$-homogeneous functors of Theorem \ref{thm 1} we further characterise the $n$-homogeneous model structure, similarly to \cite[Theorem 10.1]{BO13}. We give an alternative description of the weak equivalences as those detected by the homotopy fibre of the map $T_nF \longrightarrow T_{n-1}F$ rather than equivalences detected by the derivative, and further characterise the acyclic fibrations and cofibrations. These characterisations allow for the theory of localisations to be used computationally in the theory.

The $n$-homogeneous model structure is zig-zag Quillen equivalent to spectra with an action of $\U(n)$. The zig-zag equivalence moves through an intermediate category, $\U(n)\E_n$. This category behaves like spectra, but with structure maps $S^{2n} \wedge X_k  \longrightarrow X_{k+1}$, and is the natural home for the $n$-th derivative of a unitary functor. It comes with a stable model structure, called the $n$-stable model structure, which is an alteration of the stable model structure on spectra to take into account the unusual structure maps. The zig-zag Quillen equivalence is proved in two steps. The first, which appears as Theorem \ref{QE: intermediate and spectra}, demonstrates a Quillen equivalence between the intermediate category and the category of unitary spectra with an action of $\U(n)$. 

\begin{alphtheorem}
Let $n$ be a non-negative integer. There is an adjoint pair 
\[
\adjunction{(\alpha_n)_!}{\U(n)\E_n}{\s^\mathscr{U}[\U(n)]}{(\alpha_n)^*},
\]
which is a Quillen equivalence.
\end{alphtheorem}

The second proves that the $n$-homogeneous model structure, $n\homog\E_0$, is Quillen equivalent to the intermediate category, $\U(n)\E_n$. This is Theorem \ref{n-homog and intermediate QE} in the text.

\begin{alphtheorem}
Let $n$ be a non-negative integer. There is an adjoint pair
\[
\adjunction{\res_0^n/\U(n)}{\U(n)\E_n}{n\homog\E_0}{\ind_0^n\varepsilon^*},
\]
which is a Quillen equivalence.
\end{alphtheorem}

Moreover, we exhibit a Quillen equivalence between unitary spectra with an action of $\U(n)$ and orthogonal spectra with an action of $\U(n)$, via a Quillen equivalence between orthogonal and unitary spectra. We prove these results as Theorem \ref{thm: QE of models of spectra} and Corollary \ref{cor: spectra QE}. 

\begin{alphtheorem}
Let $n$ be a non-negative integer. There are adjoint pairs 
\[
\adjunction{r_!}{\s^\mathscr{U}}{\s^\mathscr{O}}{r^*}\ \ \text{and} \ \ \adjunction{r_!}{\s^\mathscr{U}[U(n)]}{\s^\mathscr{O}[U(n)]}{r^*},
\]
which are Quillen equivalences.
\end{alphtheorem}

This gives a complete picture of the model categories for unitary calculus and their relations. 
\[
\xymatrix@C+1cm{
n\homog\E_0 \ar@<-1ex>[r]_(.55){\ind_0^n\varepsilon^*} &  \U(n)\E_n \ar@<-1ex>[l]_(.4){\res_0^n/\U(n)} \ar@<1ex>[r]^{(\alpha_n)_!} & \s^\mathscr{U}[\U(n)] \ar@<1ex>[l]^{(\alpha_n)^*} \ar@<1ex>[r]^{r_!} & \ar@<1ex>[l]^{r^*} \s^\mathscr{O}[\U(n)]. 
}
\]

The added complexity in dealing with complex inner product spaces results in an extra adjunction than in the orthogonal case \cite[Proposition 8.3, Theorem 10.1 ]{BO13}.

\subsection*{Convergence}
One gap in the literature with regards to both orthogonal and unitary calculus is the notion of agreement and analyticity. These notions are central to convergence and play an important role in Goodwillie calculus \cite{Go91, Go03}. We define the notion of weak polynomiality, Definition \ref{def: weak poly}, in this context and show that when a functor is weakly polynomial its Taylor tower converges, that is, weakly polynomial functors are weakly analytic and they allow for more straightforward computations. 

We generalise a result of Barnes and Eldred \cite[Theorem 4.1]{BE16} to the setting of weak polynomial functors. This is provided as Theorem \ref{quasi thm} in the text. 

\begin{alphtheorem} 
Let $E, F\in \E_0$ be such that there is a homotopy fibre sequence 
\[
E(U) \longrightarrow F(U) \longrightarrow  F(U \oplus V)
\]
for $U,V \in \J$. Then 
\begin{enumerate}
\item If $F$ is weakly polynomial, then $E$ is weakly polynomial; and
\item If $E$ is weakly polynomial and $F(U)$ is $1$-connected whenever $\dim(U) \geq \rho$, then $F$ is weakly polynomial.
\end{enumerate}
\end{alphtheorem}

With this Theorem we prove that the functor $\BU(-) : V \longmapsto \BU(V)$ converges to $\BU(V)$ for $V$ with $\dim(V) \geq 1$. 

In Theorem \ref{representable functors are analytic} we show that any representable functor is weakly polynomial, and hence its associated Taylor tower converges. 

\begin{alphtheorem}
Representable functors are weakly polynomial, that is, for all $V \in \J$, the functor $\J(V,-)$ is weakly polynomial. 
\end{alphtheorem}

\subsection*{Organisation} 
In Section \ref{section: polynomial functors} we define polynomial functors, construct the polynomial approximation functor and show how this data assembles into a Taylor tower.

Section \ref{section: derivatives} concerns the derivatives of unitary functors. We also construct the $n$-homogeneous model structure. An initial step toward the zig-zag Quillen equivalence is complete in Section \ref{section: intermediate category} where we define the intermediate category for unitary calculus and prove a Quillen equivalence between the intermediate category and the $n$-homogeneous model structure. 

We give a description of homotopy theory of the derivatives through a Quillen equivalence between the intermediate category and spectra with an action of $\U(n)$ in Section \ref{section: intermediate category as spectra}. In Section \ref{section: diff as a QF} we prove that the differentiation functor is a right Quillen functor as part of a Quillen equivalence. 

In Section \ref{section: classification of homogeneous functors} we give a classification of $n$-homogeneous functors in terms of spectra with an action of $\U(n)$. We also give further characterisations of the $n$-homogeneous weak equivalences, acyclic fibrations and cofibrations.  This section ends with a short discussion on the complete Taylor tower. 

The final section, Section \ref{section: convergence}, is an initial step toward understanding convergence of the Taylor tower in unitary calculus.

\subsection*{Conventions and Notation} Throughout this paper the category of based compactly generated weak Hausdorff topological spaces will be denoted $\T$. We endow $\T$ with the Quillen model structure, this is cofibrantly generated with the set of generating cofibrations denoted $I$ and generating acyclic cofibrations $J$. 

In this paper we make strong use of the theory of model categories and Bousfield localisations. These provide us with the tools to make precise the ``up to homotopy'' results of Weiss. We refer the unfamiliar reader to \cite{DS95, Ho99} and \cite{Hi03} for the details of the theory. 

\subsection*{Acknowledgments} This work forms part of the authors Ph.D. project under the supervision of David Barnes. The author wishes to thank him for thoughtful and enthusiastic discussions. The author also thanks Greg Arone for helpful comments. We further thank the referee for numerous helpful comments which greatly improved the quality and exposition of this work.

\section{Polynomial Functors and the Taylor Tower}\label{section: polynomial functors}

As with differential calculus, the building blocks of functor calculus are the polynomial functors. In particular the polynomial approximations are the pieces that fit together into the Taylor tower, and in some cases converge to the original input functor. In this section we give the constructions of polynomial functors and polynomial approximations. We highlight the strong analogy between unitary calculus and differential calculus and how this machinery produces a Taylor tower.

\subsection{The input functors} We start with a discussion on the input functors. Let $\J$ be the category of finite-dimensional complex inner product subspaces of $\C^\infty$, with morphism the complex linear isometries. This category is $\T$-enriched, the space of morphisms $\J(U,V)$ is the Stiefel manifold of $\dim (U)$-frames in $V$. Let $\J_0$ be the category with the same objects as $\J$ and morphism space given by $\J_0(U,V) = \J(U,V)_+$. We denote the category of $\T$-enriched functors from $\J_0$ to $\T$  by $\E_0$. This is the input category for unitary calculus. Examples of such functors are abound. These include;
\begin{enumerate}
\item $\mathsf{BU}(-): V \longmapsto \mathsf{BU}(V)$, where $\mathsf{BU}(V)$ is the classifying space of the unitary group associated to $V$;
\item $\mathsf{BTOP}(-): V \longmapsto \mathsf{BTOP}(V)$, where $\mathsf{BTOP}(V)$ is the classifying space of the group of homeomorphisms on $V$; and
\item $\mathsf{B}\mathcal{G}(\mathbb{S}(-)): V \longmapsto \mathsf{B}\mathcal{G}(S^V)$, where $\mathsf{B}\mathcal{G}(S^V)$ is the classifying space of the group-like monoid of homotopy equivalences between $S^V$ and itself.
\end{enumerate}

These examples all have orthogonal counterparts, see \cite{We95}. This indicates a strong relationship between the orthogonal and unitary calculi, induced by the strong relationship between functors of this type and their orthogonal versions. A deeper understanding of the unitary calculus framework will lead to a better understanding of this relationship. 

The category $\E_0$ is a diagram category in the sense of Mandell, May, Schwede and Shipley \cite{MMSS01} and hence comes equipped with a projective model structure, where the weak equivalences and fibrations are defined to be the levelwise weak homotopy equivalences and levelwise Serre fibrations respectively. This a cellular proper topological model category with generating (acyclic) cofibrations of the form $\J_0(V,-) \wedge i$ with $i$ a generating (acyclic) cofibration of $\T$, see \cite[Theorem 6.5]{MMSS01}.

\subsection{Polynomial functors} In differential calculus, polynomial functions are used to approximate a given function. These polynomial functions may be used to gain insight about the original function and can be combined in such a way to give a complete approximation of the function. Polynomial functors are a categorification of this idea. We begin with the definition. 

\begin{definition}\label{definition: polynomial functor}
Let $n$ be a non-negative integer. A functor $F$ in $\E_0$ is \textit{polynomial of degree less than or equal to $n$} or equivalently \textit{$n$-polynomial} if the canonical map 
\[
\rho: F(V) \longrightarrow \underset{0 \neq U \subseteq \C^{n+1}}{\holim} F(V \oplus U) =: \tau_n F(V),
\]
is a weak homotopy equivalences of spaces. 
\end{definition}

The poset of non-zero complex subspaces of $\C^{n+1}$ is a category object in spaces, specifically, a category internal to $\sf{Top}$. The space of objects is topologised as a disjoint union of complex Grassmannian manifolds, and the space of morphisms topologised as the space of complex flag manifolds of length two (or equivalently the space of $2$-simplices in the nerve of the poset). As such, the homotopy limit of Definition \ref{definition: polynomial functor} must take into account the topology of the poset. The details may be found in \cite{We95, We98} for the orthogonal case. The unitary case follows similar. Hollender and Vogt \cite{HV92} and Lind \cite{Lin09} give constructions for homotopy limits and colimits indexed on categories internal to spaces. 

\begin{rem}
In comparison to Goodwillie calculus, $F : \mathcal{C} \longrightarrow \mathcal{C}$ (where $\mathcal{C}$ is some appropriate $(\infty,1)$-category, for instance $\mathcal{C}=\T$) is $1$-excisive (linear) it if takes (homotopy) pushouts to (homotopy) pullbacks, \cite[Definition 1.2]{AC19}. In this situation the homotopy limit is indexed on the poset of subsets of the set of two elements, \cite[Definition 1.1]{AC19}.  Our situation is significantly more complicated due to the topology involved in the indexing poset $\{0 \neq U \subseteq \C^2\}$ where there is an $\C P^{1}$ ($\cong S^2$) worth of one-dimensional complex subspaces. More generally the higher dimensional cubes of Goodwillie translate to a disjoint union of Grassmannian manifolds in our setting. Note this is also true for orthogonal calculus, where there is an $\R P^1 \cong S^1$ worth of one-dimensional subspaces. Figure \ref{fig 1} illustrates the poset for linear functors in orthogonal calculus (on the left) and for Goodwillie calculus (on the right).

\begin{figure}[h]
\begin{center}
\begin{tikzpicture}[scale=2.5]
\draw[thick] (1,0) ellipse (0.25cm and 0.5cm);
\node[label={$\mathbb{R}^2$}] (a) at (2,0) {$\bullet$};
\draw (1,0.5) -- (2,0);
\draw (1.1,-0.46) -- (2,0);
\draw(1.22,-0.2) -- (2,0);
\draw [dashed](0.75,0.15) -- (2,0);
\draw [
    thick,
    decoration={
        brace,
        mirror,
        raise=0.5cm
    },
    decorate
] (0.8,-0.4) -- (1.2, -0.4);
\node [scale=0.8] at (1, -0.8) {$1$-dim subspaces};
\node (b) at (4,0.25) {$\{1\}$};
\node (c) at (4,-0.5) {$\{1,2\}$};
\node (d) at (3.25,-0.5) {$\{2\}$};
\draw[->] (b) -- (c);
\draw[->] (d) -- (c); 
\end{tikzpicture}
\end{center}
\caption{Poset schematics for linear functors in orthogonal and Goodwillie calculus.}
\label{fig 1}
\end{figure}
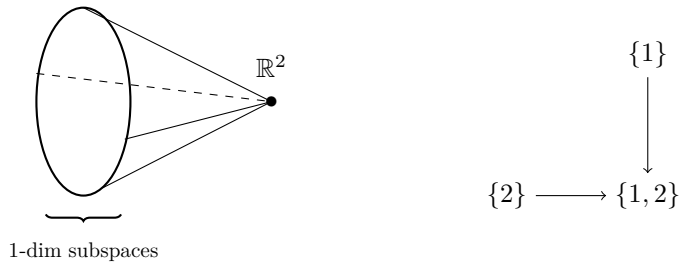
\end{rem}

\begin{ex}
A functor $F$ is polynomial of degree zero if and only if $F$ is homotopically constant. Indeed, if $F$ is $0$-polynomial then, $F(V) \simeq F(U \oplus \C)$ for all $U \in \J_0$. Iterating this, we see that if $F$ is $0$-polynomial then, $F(U) \simeq F(U \oplus V)$ for all $U,V \in \J_0$. For the converse, if $F$ is homotopically constant, $F(U) \simeq F(U \oplus \C)$ and hence the condition of Definition \ref{definition: polynomial functor} is satisfied. 
\end{ex}

\subsection{Polynomial approximation} 

If $E$ is $n$-polynomial then the canonical map $\rho: E \longrightarrow \tau_nE$ is a levelwise weak equivalence, and hence all the morphisms in the diagram 
\[
\xymatrix{
E \ar[r]^{\rho} &  \tau_nE \ar[r]^{\rho} & \cdots \ar[r]^\rho &  \tau_n^k E \ar[r]^{\rho} & \cdots
}
\]
are levelwise weak equivalences. Moreover by construction, the $k$-th iterate $\tau_n^k$ of the functor $\tau_n$ is determined by its behaviour on vector spaces of dimension at least $k$, since 
\[
\tau_n^k F(0) = \underset{0 \neq U_1, \cdots, U_k \subseteq \C^{n+1}}{\holim} F(U_1 \oplus \cdots U_k).
\]
Polynomial functors are said to be ``determined by their behaviour at infinity''. This property motivates the following definition. 

\begin{definition}\label{definition: polynomial approximation}
Let $n$ be a non-negative integer. Define the \textit{$n$-th polynomial approximation}, $T_nF$, of a unitary functor $F$ to be the homotopy colimit of the filtered diagram
\[
\xymatrix{
F \ar[r]^{\rho} &  \tau_nF \ar[r]^{\rho} &  \tau_n^2F \ar[r]^{\rho} &  \tau_n^3F \ar[r]^{\rho} & \cdots.  
}
\]
\end{definition}

\begin{rem}
Equivalently we could have defined $T_nF$ to be the homotopy colimit of the filtered diagram
\[
\xymatrix{
F \ar[r]^{\rho} &  \tau_nF \ar[r]^{\tau_n(\rho)} &  \tau_n^2F \ar[r]^{\tau_n^2(\rho)} &  \tau_n^3F \ar[r]^{\tau_n^3(\rho)} & \cdots.  
}
\]
The direct systems are weakly equivalent and hence define isomorphic homotopy colimits. 
\end{rem}

\begin{ex}
The zeroth polynomial approximation of a unitary functor $F$ is $F(\C^\infty)$. Indeed, it is clear that 
\[
T_0 F(V) = \hocolim_k \tau_0^kF(V) = \hocolim_k F(V \oplus \C^k) \simeq F(\C^\infty). 
\]
\end{ex}

In particular, we see that the zeroth polynomial approximation of the sphere functor $\mathbb{S} : V \longmapsto S^V$ is $S^\infty \simeq \ast$, the zeroth polynomial approximation of $\mathsf{BU}(-)$ is the space $\mathsf{BU}$, and the zeroth polynomial approximation of $\U(-): V \longmapsto \U(V)$ is the infinite unitary group $\U$. 

The Taylor polynomials $p_n(x)$ of a function $f: \R \longrightarrow \R$ are in some ways the closest polynomial functions to $f$. In particular if $f$ is $n$-polynomial, then $f(x) = p_n(x)$. This result has a categorification in that if $F$ is $n$-polynomial then $F \simeq T_nF$. 

\begin{prop}\label{equiv to poly approx}
Let $n$ be a non-negative integer. If $F$ is $n$-polynomial, then the canonical map $\eta : F \longrightarrow T_nF$ is a levelwise weak homotopy equivalence.
\end{prop}
\begin{proof}
Since $F$ is $n$-polynomial, $\rho : F \longrightarrow \tau_n F$ is a levelwise weak equivalence. Since finite homotopy limits (particularly those defining $\tau_nF$) commute, $\tau_nF$ is $n$-polynomial when $F$ is. It follows by the properties of homotopy colimits that $\eta : F \longrightarrow T_nF$ is a levelwise weak equivalence. 
\end{proof}

The $n$-th polynomial approximation is the closest $n$-polynomial functor to $F$, that is, if $\nu: F \longrightarrow E$ is a map in $\E_0$, and $E$ is $n$-polynomial, then $\nu$ factors (up to homotopy) through the map $\eta: F \longrightarrow T_nF$. Possibly the most straightforward way to see that $T_nF$ is the closest $n$-polynomial functor to $F$ is the construction of a model structure of which $T_nF$ is a model for the fibrant replacement of $F$. The construction of such a model structure is similar to that of Barnes and Oman \cite[Proposition 6.5, Proposition 6.6]{BO13} in the orthogonal calculus case. The first step is showing that $T_nF$ is always $n$-polynomial, which follows readily from a unitary version of \cite{We98}.

\begin{lem}\label{Tn is n-polynomial}
Let $n$ be a non-negative integer. If $F  \in \E_0$, then $T_nF$ is $n$-polynomial. 
\end{lem}

We now give the $n$-polynomial model structure. 

\begin{prop}\label{prop: n-poly model structure}
Let $n$ be a non-negative integer. There is a cellular proper topological model structure on $\E_0$ where a map $f: E \longrightarrow F$ is a weak equivalence if $T_nf: T_nE \longrightarrow T_nF$ is a levelwise weak equivalence, the cofibrations are the cofibrations of the projective model structure and the fibrations are levelwise fibrations such that 
\[
\xymatrix{
E \ar[r]^f \ar[d]_{\eta_X} & F \ar[d]^{\eta_Y} \\
T_nE \ar[r]_{T_nf} & T_n F
}
\]
is a homotopy pullback square. The fibrant objects of this model structure are precisely the $n$-polynomial functors and $T_n$ is a fibrant replacement functor. We call this the $n$-polynomial model structure and it is denoted $n\poly\E_0$.
\end{prop}
\begin{proof}
The Bousfield-Friedlander localisation \cite{BF78, Bo01} of the projective model structure at the endofunctor $T_n : \E_0 \longrightarrow \E_0$ yields the stated model structure.  Note however, that the Bousfield-Friedlander localisation only results in a proper topological model structure. An alternative description as the left Bousfield localisation of the projective model structure at the set of maps 
\[
\mathcal{S}_n = \{ S\gamma_{n+1}(V,-)_+ \longrightarrow \J_0(V,-) \ : \ V \in \J_0\},
\]
yields the cellular requirement. These two descriptions agree since both localisation techniques do not alter the cofibrations and a $T_n$-equivalence in the sense of \cite{BF78, Bo01} is precisely a $\mathcal{S}_n$-local equivalence in the sense of \cite[Definition 3.1.4]{Hi03}.
\end{proof}

Polynomial functors share many properties with polynomial functions. One such property is that an $n$-polynomial functor is $(n+1)$-polynomial. 

\begin{prop}\label{n poly is (n+1) poly}
Let $n$ be a non-negative integer. If a functor $F$ is $n$-polynomial, then it is $(n+1)$-polynomial.
\end{prop}
\begin{proof}
This is the unitary version of \cite[Proposition 5.4]{We95}. The properties of real vector bundles used by Weiss transfer to complex vector bundles, and hence so does the result. 
\end{proof}

We will return to further properties of $n$-polynomial functors once we have introduced the notion of the derivative of a functor. 

\subsection{The Taylor Tower} With the theory of polynomial functors and polynomial approximations in place, we can construct the Taylor tower approximating $F \in \E_0$. This tower is a categorification of the Taylor series associated to a function $f: \R \longrightarrow \R$. The differential Taylor's series is a sequence of polynomial approximations which converges to the function $f$. In this case, we get a tower of polynomial approximations, the limit of which - in nice cases - recovers the input functor. We return to the question of convergence in Section \ref{section: convergence}.

There is an inclusion functor from the poset of non-zero subspaces of $\C^{n-1}$ to the poset of non-zero subspaces of $\C^n$. For a functor $F \in \E_0$, precomposition with the inclusion defines a map 
\[
\tau_nF(V) = \underset{0 \neq U \subseteq \C^n}{\holim}F(U \oplus V) \longrightarrow \underset{0 \neq U \subseteq \C^{n-1}}{\holim}F(U \oplus V) = \tau_{n-1}F(V).
\]
Iteration constructs a map $\tau_n^k F \longrightarrow \tau_{n-1}^k F$. The construction is coherent in that the diagram 
\[
\xymatrix{
\underset{0 \neq U_1, \cdots, U_k \subseteq \C^n}{\holim}F(U_1 \oplus \cdots \oplus U_k \oplus V) \ar[r] \ar[d]_{\rho} &  \underset{0 \neq U_1, \cdots, U_k \subseteq \C^{n-1}}{\holim}F(U_1 \oplus \cdots \oplus U_k \oplus V)  \ar[d]^{\rho}\\
\underset{0 \neq U_1, \cdots, U_{k+1} \subseteq \C^n}{\holim}F(U_1 \oplus \cdots \oplus U_{k+1} \oplus V)\ar[r] &  \underset{0 \neq U_1, \cdots, U_{k+1} \subseteq \C^{n-1}}{\holim}F(U_1 \oplus \cdots \oplus U_{k+1} \oplus V) 
}
\]
commutes. As such, we get a map of homotopy colimits, $r_n : T_nF \longrightarrow T_{n-1}F$. Moreover $r_n \eta_n = \eta_{n-1}$, where $\eta_n : F \longrightarrow T_nF$. The result is a Taylor tower of the following form. 
\[
\xymatrix@C+1cm{
		&			&	\ar@/_1pc/[dl]	F \ar[d]	  \ar@/^1pc/[drr]  \ar@/^1.3pc/[drrr]   &			     	&			& \\
 \cdots \ar[r] & T_{n+1}F \ar[r]_{r_{n+1}} & T_nF \ar[r]_{r_n} & \cdots \ar[r]_{r_2} & T_1F \ar[r]_{r_1} & F(\C^\infty) 
}
\]

\subsection{Homogeneous functors} The $n$-th layer of the Taylor tower is given by the homotopy fiber of
\[
r_n : T_nF \longrightarrow T_{n-1}F.
\]
This functor is  both $n$-polynomial and its $(n-1)$-st polynomial approximation is trivial. 

\begin{example}
Let $F$ be a unitary functor and $n$ a positive integer. The homotopy fibre
\[
D_nF = \hofibre[T_nF \longrightarrow T_{n-1}F],
\]
is $n$-polynomial and $T_{n-1}D_nF$ is trivial. First note that by using the fact that homotopy limits commute, the long exact sequence of a fibration and the Five Lemma, that the homotopy fibre of a map between $n$-polynomial functors is $n$-polynomial. Since an $(n-1)$-polynomial object is $n$-polynomial and the homotopy fibre of a map between $n$-polynomial objects is $n$-polynomial, $D_nF$ is $n$-polynomial. Moreover, the $(n-1)$-st polynomial approximation of $D_nF$ is trivial since
\[
\begin{split}
T_{n-1}D_n F &= T_{n-1} \hofibre[T_nF \longrightarrow T_{n-1}F] \simeq \hofibre[T_{n-1}T_nF \longrightarrow T_{n-1}^2 F] \\ 
&\simeq \hofibre[T_nT_{n-1}F \longrightarrow T_{n-1}F]\simeq \hofibre[T_{n-1}F \longrightarrow T_{n-1}F] \simeq \ast.
\end{split}
\]
\end{example}

Functors with these properties are a special subclass of $n$-polynomial functors, called $n$-homogeneous functors. One should think of these homogeneous functors as monomial functions, representing the terms in the Taylor series from differential calculus. 

\begin{definition}\label{definition: homogeneous functor}
Let $n$ be a non-negative integer. A unitary functor $F$ is said to be \textit{$n$-homogeneous} or equivalently \textit{homogeneous of degree less than or equal $n$} if $F$ is $n$-polynomial and $T_{n-1}F$ is levelwise weakly contractible. We will refer to a functor with trivial $(n-1)$-st polynomial approximation as being \textit{$n$-reduced}. 
\end{definition}

\section{The derivative of a functor}\label{section: derivatives}

An important concept in any theory of calculus is that of derivative. In the differential setting, the derivatives are used in calculating the terms of the polynomial approximations. The same remains true in the unitary calculus setting, where the derivatives (or more precisely the spectrum formed by the derivatives) are used to characterise the layers of the Taylor tower. 

The construction of the derivative is completely analogous to that of Weiss \cite{We95} for the derivative of an orthogonal functor. In particular, the $n$-th derivative of a unitary functor determines a unitary spectrum with an action of $\U(n)$. 

\subsection{Derivatives} From the derivatives of a unitary functor we can construct a unitary spectrum, and hence an orthogonal spectrum. We first define the derivative of a functor, and highlight the justification for this construction being called a derivative, specifically, homotopically constant functors have trivial derivative, and the $(n+1)$-st derivative of an $n$-polynomial functor is trivial. The starting point is the construction of ``higher'' versions of $\J_1$, specifically, categories $\J_n$ for all $n$. Sitting over the space of linear isometries $\J(U,V)$ is the $n$-th complement vector bundle, with total space
\[
\gamma_n(U,V) = \{ (f,x) \ : \ f \in \J(U,V), x \in \C^n \otimes f(U)^\perp\},
\]
where $f(U)^\perp$ denotes the orthogonal complement of $f(U)$ in $V$. 

The vector bundle $\gamma_n(U,V)$ comes with a sphere bundle $S\gamma_n(U,V)$ given by the one-point compactification of the fibres. In \cite[Theorem 4.1, Proposition 4.2]{We95} Weiss constructs a homeomorphism between the sphere bundle and a particular homotopy colimit. To prove the equivalence between the intermediate category and the $n$-homogeneous model structure we need a similar result, which we give here. For the proof, we will denote by $\mathcal{C}$ the poset of non-zero subspaces of $\C^{n+1}$. 

\begin{thm}\label{hocolim sphere homeo}
Let $n$ be a non-negative integer. There is a natural homeomorphism, 
\begin{equation*}
\underset{0 \neq U \subset \C^{n+1}}{\hocolim}  \J(U \oplus V, W) \longrightarrow S\gamma_{n+1}(V,W),
\end{equation*}
for every $V,W \in \J$.
\end{thm}
\begin{proof}
The goal is to construct a homeomorphism
\begin{equation*}
\xymatrix{
\Psi : \J_{n+1}(V,W) \setminus \J(V,W)_+ \ar[r] & (0, \infty) \times \hocolim_U \J(U \oplus V, W),
}
\end{equation*}
and use the identification, $\J_{n+1}(V,W) \setminus \J(V,W)_+ \cong (0,\infty) \times S\gamma_{n+1}(V,W)$, to yield a homeomorphism
\begin{equation*}
\underset{0 \neq U \subset \C^{n+1}}{\hocolim}  \J(U \oplus V, W) \longrightarrow S\gamma_{n+1}(V,W).
\end{equation*}

Take $(f,x) \in \J_{n+1}(V,W) \setminus \J(V,W)_+$, that is $f \in \J(V,W)$ and $0\neq x \in (n+1) \cdot (f(V)^\perp)$. Since $f(V)^\perp$ and $\C^{n+1}$ are finite dimensional inner product spaces we have
\begin{equation*}
\begin{split}
\C^{n+1} \otimes (f(V)^\perp) \cong (\C^{n+1})^* \otimes (f(U)^\perp) &= \Hom(\C^{n+1}, \C) \otimes f(U)^\perp \\
&\cong \Hom(\C^{n+1}, f(U)^\perp),
\end{split}
\end{equation*}
where $(\C^{n+1})^*$ is the dual space of $\C^{n+1}$. 

Thus we can think of $x \in \Hom(\C^{n+1}, f(V)^\perp)$. Moreover, $x$ has an adjoint 
\[
x^* : f(V)^\perp \longrightarrow \C^{n+1}.
\]  

By the properties of the dual transformation it follows that
\[
(x^*x)^* = x^*(x^*)^* = x^*x,
\]
that is, $x^*x : \C^{n+1} \longrightarrow \C^{n+1}$ s a self adjoint map, hence normal, in particular, $(x^*x)^*(x^*x)= (x^*x)(x^*x)^*$. By the Spectral Theorem, see for instance Friedberg, Insel, and Spence \cite[Theorem 6.24]{FIS89}, $\C^{n+1}$ can be written as the direct sum of the eigenspaces corresponding to the distinct eigenvalues of $x^*x$, that is, 
\begin{equation*}
\C^{n+1} = \ker(x^*x) \oplus E(\lambda_0) \oplus \dots \oplus E(\lambda_k)`
\end{equation*}
where $0 < \lambda_0 < \lambda_1 < \dots < \lambda_k$ are the eigenvalues, and 
\begin{equation*}
E(\lambda_i) =\{ r \in \C^{n+1} \ : \ (x^*x - \lambda_i   \id)(r)=0 \}.
\end{equation*}
is the eigenspaces corresponding the eigenvalue $\lambda_i$. Note that all the eigenvalues are real by \cite[Lemma pg.329]{FIS89}, and by the Spectral Theorem \cite[Theorem 6.24]{FIS89} we really mean the eigenspace, and not the generalised eigenspace. 

Hence, given $(f,x) \in \J_{n+1}(V,W) \setminus \J(V,W)_+$ as above, define
\begin{description}
\item[1)] a functor $G : [k] \longrightarrow \mathcal{C}$ given by 
\begin{equation*}
r \longmapsto E(\lambda_0) \oplus \dots \oplus E(\lambda_{k-r})
\end{equation*}

\item[2)] a linear isometry $z \in \J(G(0) \oplus V, W)$ given by 
\begin{equation*}
z = \begin{cases}
f \ \text{on} \ V \\
\lambda_i^{-1/2}\cdot x \ \text{on} \ E(\lambda_i)
\end{cases}
\end{equation*}
Note that $z$ is clearly an isometry on $V$, and $z$ is also an isometry on each $E(\lambda_i)$. Indeed, let $v,w \in E(\lambda_i)$. Then
\[
\begin{split}
\langle z(v), z(w) \rangle &= \langle (\lambda_i)^{-1/2}x(v), (\lambda_i)^{-1/2}x(w) \rangle \\
&= (\lambda_i)^{-1} \langle x(v), x(w) \rangle \\
&= (\lambda_i)^{-1} \langle v, x^*x(w) \rangle \\ 
&= (\lambda_i)^{-1} \langle v, \lambda_i w \rangle\\ 
&= \langle v, w \rangle.
\end{split}
\]
\item[3)] a point $p \in \Delta^k$ given by the barycentric coordinates
\begin{equation*}
\lambda_k^{-1} \cdot (\lambda_0, \lambda_1 - \lambda_0, \lambda_2-\lambda_1, \dots, \lambda_k - \lambda_{k-1})
\end{equation*}
\item[4)] $t=\lambda_k >0$.
\end{description}
By construction (see \cite[\S VIII.2.6]{BK72}), $\hocolim_U \J(U \oplus V, W)$ (or more specifically, the geometric realisation of the bar construction) is a quotient of 
\begin{equation*}
\coprod_{k \geq 0} \coprod_{G:[k] \longrightarrow \mathcal{C}} \J(G(0) \oplus V, W) \times \Delta^k
\end{equation*}
under the identifications involving the respective face and degeneracy maps for 
\[
\coprod_{G:[k] \longrightarrow \mathcal{C}} \J(G(0) \oplus V, W)
\]
and $\Delta^k$ respectively, again, see \cite[\S VIII.2.6]{BK72} for the details. It follows that the triple $(G,z,p)$ defines a point in $\hocolim_U \J(U \oplus V, W)$. 

Define a map 
\begin{equation*}
\begin{split}
\Psi : \J_{n+1}(V,W) \setminus \J(V,W)_+ &\longrightarrow (0, \infty) \times \hocolim_U \J(U \oplus V, W)\\
(f,x) &\longmapsto (t, G, z, p)
\end{split}
\end{equation*}
This map is well defined since $(G,z,p)$ defines a point in $\hocolim_U \J(U \oplus V, W)$.

\underline{$\Psi$ is injective:}
Let $(f,x),(g,y) \in \J_{n+1}(V,W) \setminus \J(V,W)_+$ be such that 
\[
\Psi(f,x) = (t,G,z,p) = (t', G', z', p')= \Psi(g,x).
\]

Since $t=t'$, $x^*x$ and $y^*y$ must have the same largest eigenvalue, and $p=p'$ together with $x^*x$ and $y^*y$ having the same largest eigenvalue implies that $x^*x$ and $y^*y$ have the same set of non-zero eigenvalues $\lambda_0 < \lambda_1 < \dots < \lambda_k$. Combining this with the fact that $G=G'$, we have that  $x^*x$ and $y^*y$ must have the same eigenspaces. Moreover, since $z=z'$, $z|_V = Z'|_V$, and hence, $f=g$. Moreover $z|_{G(0)} = z'|_{G(0)}$.  It follows that \mbox{$z|_{E(\lambda_i)} = z'|_{E(\lambda_i)}$}, for every $i$, and by definition, we see that $x=y$ and hence that $(f,x)=(g,y)$ as required. 

\underline{$\Psi$ is onto:} Here there are two cases to consider. If $k \geq n+1$ we are in the degenerate case.
Otherwise pick $t>0$ and let $\lambda_k=t$. Now choose a point in the interior of $\Delta^k_+$. If we were to choose a point in the boundary of $\Delta^k_+$, we would be choosing a point in a face of $\Delta^k_+$, and faces are given by $\Delta^{k-1}_+$ and smaller. Thus choice of a point determines 
\[
\lambda_0 < \lambda_1 < \dots < \lambda_k \in \C.
\]
Choose orthogonal eigenspaces for each $\lambda_i$, and define $G: [k] \longrightarrow \Delta$ as above, using these chosen eigenspaces. 

Next choose any $z \in \J(G(0) \oplus V, W)$. Then define $f = z|_V$ and $x = (\lambda_i)^{1/2}z|_{E(\lambda i)}$ on $E(\lambda_i)$. Then $x : \C^{n+1} \longrightarrow f(V)^\perp$, and $\Psi(f,x) = (t,G,z,p)$ as required. 

\underline{$\Psi$ is continuous:} This is clear from the construction of $\Psi$ and from the fact that polynomials are continuous. Note that the functor $G$ depends continuously on $(f,x)$ as the eigenvalues and hence eigenspaces depend continuously on $x$.

\underline{$\Psi$ is a homeomorphism:} The reverse of the construction used to show $\Psi$ is onto, gives the inverse to $\Psi$. For the same reasons as $\Psi$ is continuous, this inverse is continuous. 

We thus have 
\begin{equation*}
\Psi : \J_{n+1}(V,W) \setminus \J(V,W)_+ \xlongrightarrow{\cong} (0, \infty) \times \hocolim_U \J(U \oplus V, W)
\end{equation*}
Notice that 
\begin{equation*}
\J_{n+1}(V,W) \setminus \J(V,W)_+ \cong (0,\infty) \times S\gamma_{n+1}(V,W),
\end{equation*}
and hence 
\begin{equation*}
 (0, \infty) \times \hocolim_U \J(U \oplus V, W) \cong (0,\infty) \times S\gamma_{n+1}(V,W).
\end{equation*}
It follows by composing with the inclusion 
\[
\{1\} \times \hocolim_U \J(U \oplus V, W) \longrightarrow (0, \infty) \times \hocolim_U \J(U \oplus V, W),
\]
and projecting to $S\gamma_{n+1}(V,W)$ that 
\begin{equation*}
 \hocolim_U \J(U \oplus V, W) \cong S\gamma_{n+1}(V,W).
\end{equation*}
\end{proof}

Theorem \ref{hocolim sphere homeo} is arguably the most important result in unitary calculus. It firstly allows us to convert the notion of $n$-polynomial into geometric terms, as in \cite[Proposition 5.2]{We95}, and secondly it allows us to argue that the sphere bundle $S \gamma_n(U,V)_+$ is cofibrant in $\E_0$.

\begin{rem}
Although Theorem \ref{hocolim sphere homeo} looks similar to \cite[Proposition 4.2]{We95} there are subtle differences from the differences between complex and real linear algebra. One such difference is the results used to exhibit real eigenvalues. 
\end{rem}

\begin{definition}
Let $n$ be a non-negative integer. The \textit{$n$-th jet category} $\J_n$ is the category with the same objects as $\J$ and with morphism space $\J_n(U,V)$ the Thom space, $\mathrm{Th}(\gamma_n(U,V))$, of the vector bundle $\gamma_n(U,V)$. 
\end{definition}

For $0\leq m < n$, the inclusion of $\C^m$ into $\C^n$ onto the first $m$-components induces a functor $i_m^n : \J_m \longrightarrow \J_n$. Precomposition with such determines a functor which restricts from $\E_n$ to $\E_m$. 

\begin{definition}\label{def: derivative}
Let $0\leq m < n$. Define the \textit{restriction functor} $\res_m^n : \E_n \longrightarrow \E_m$ by precomposition with $i_m^n$, and define the \textit{induction functor} $\ind_m^n : \E_m \longrightarrow \E_n$ to be the right Kan extension along $i_m^n$. Given a functor $F \in \E_0$ we will call $\ind_0^n F$ its $n$-th derivative and denote this by $F^{(n)}$. 
\end{definition}

Using the properties of the adjunction 
\[
\adjunction{\res_m^n}{\E_n}{\E_m}{\ind_m^n},
\]
and the Yoneda Lemma we see that $\ind_m^nF(U) \cong \E_m(\J_n(V,-), F)$. For the purposes of calculations there is a more useful description. 

\begin{prop}[Proposition B]\label{derivatives stable}
Let $n$ be a non-negative integer. There is a homotopy fibre sequence 
\[
\res_n^{n+1}\ind_n^{n+1}F(U) \longrightarrow F(U) \longrightarrow \Omega^{2n}F(U \oplus \C),
\]
for all $F\in \E_n$ and all $U \in \J_0$. 
\end{prop}
\begin{proof}
Due to the strong similarities between complex and real linear algebra, \cite[Proposition 1.2]{We95} extends to the unitary setting. The result is a cofibre sequence 
\[
\J_n(U \oplus \C, ) -\wedge S^{2n} \longrightarrow \J_n(U,-) \longrightarrow \J_{n+1}(U,-),
\]
where we have identified the one-point compactification of $\C^n$ with the $2n$-sphere.  Applying the corepresentable functor $\E_n(-, F)$ yields a homotopy fibre sequence 
\[
\E_n(\J_{n+1}(U,-), F) \longrightarrow \E_n(\J_n(U,-), F) \longrightarrow \E_n(\J_n(U \oplus \C,-) \wedge S^{2n}, F).
\]
Combining this with the definition of the induction functor, the Yoneda Lemma and $(\Sigma, \Omega)$-adjunction yields the result. 
\end{proof}

Right Kan extensions can be constructed iteratively, particularly, 
\[
\ind_0^n = \ind_{n-1}^n \ind_{n-2}^{n-1} \cdots \ind_0^1
\]
hence Proposition \ref{derivatives stable} gives a means to iteratively calculate the derivatives.

The category $\E_1$ is equivalent to the category of spectra, see Sections \ref{section: intermediate category} and \ref{section: intermediate category as spectra}, hence the first derivative of a functor may be given in terms of a spectrum. 

\begin{example}\label{example: derivative of BU}
The first derivative of $\mathsf{BU}(-)$ is the shifted unitary sphere spectrum $\Bbb{S}^{1}$ with trivial $U(1)$ action.  The homotopy fibre sequence identifying the derivative, Proposition \ref{derivatives stable}, is precisely the homotopy fibre sequence 
\[
\Sigma S^V \longrightarrow \mathsf{BU}(V) \longrightarrow \mathsf{BU}(V \oplus \C).
\]
Hence, 
\[
\mathsf{BU}^{(1)}(V) \simeq S^{2\dim V +1}, 
\]
and thus
\[
{\mathsf{BU}}^{(1)} \simeq \Bbb{S}^{1}. 
\]
\end{example}

In differential calculus one may use the derivative of a function to determine how far a function is from being polynomial. In particular, the $(n+1)$-st derivative of an polynomial function of degree $n$ is zero. In functor calculus, it is also possible to use the derivative to measure the failure of a functor from being polynomial.

\begin{prop}[Proposition C]\label{prop: hofib of n-poly map}
Let $n$ be a non-negative integer and let $F \in \E_0$. There is a homotopy fibre sequence 
\[
F^{(n+1)}(U) \longrightarrow  F(U) \xlongrightarrow{\rho} \tau_n F(U),
\]
for all $U \in \J_0$. 
\end{prop}
\begin{proof}
This is the unitary version of \cite[Proposition 5.3]{We95}. The proof works in the exact same fashion. For an alternative perspective - which also transfers easily to the unitary case - see \cite[Lemma 5.5]{BO13}. 
\end{proof}

As a corollary we see that the $(n+1)$-st derivative of an $n$-polynomial functor is trivial.

\begin{cor}\label{cor: n+1 derivative of n-poly}
Let $n$ be a non-negative integer. If $F$ is $n$-polynomial, then $F^{(n+1)}$ is trivial.
\end{cor}

\begin{rem}
The condition of Corollary \ref{cor: n+1 derivative of n-poly} is necessary for a functor $F$ to be $n$-polynomial, but not sufficient. Consider the functor $F \in \E_0$ defined by
\[
F(U) = 
\begin{cases}
S^0 \ \text{whenever} \ \dim U > 5 \\
\ast \ \text{otherwise}.
\end{cases}
\]
and for $f \in \J_0(U,V)$, $F(f)$ is the identity. The first derivative of $F$ at $U$ is given by 
\[
F^{(1)}(U) = \hofibre[F(U) \longrightarrow F(U \oplus \C)]
\]
which is always contractible. So the first derivative of $F$ vanishes. However  for $f: \C^5 \longrightarrow \C^6$, $F(f)$ is not a weak homotopy equivalence as $\ast \not\simeq S^0$, hence $F$ is not $0$-polynomial. 
\end{rem}

A useful result for showing an object is $n$-polynomial is the following result relating $n$-polynomial objects and their homotopy fibre. The proof of which is an application of the Five Lemma.

\begin{lem}\label{hofibre lemma for polynomials} \label{hofibre of n-poly is n-poly} 
Let $n$ be a non-negative integer, and let $E \in \E_0$ be an $n$-polynomial, $g: E \longrightarrow F$ a morphism in $\E_0$ and suppose that the $(n+1)$-st derivative of $F$ is levelwise weakly contractible. Then the functor given by 
\begin{equation*}
U \longmapsto \hofibre[E(U) \xrightarrow{g_V} F(U)],
\end{equation*}
is an $n$-polynomial.
\end{lem}

We achieve the following corollary.

\begin{cor}\label{cor: delooping of n-poly}
Let $n$ be a non-negative integer and let $F \in \E_0$. If $F^{(n+1)}$ is trivial, then the functor given by $U \longmapsto \Omega F(U)$ is $n$-polynomial.
\end{cor}

An important example from orthogonal calculus \cite[Example 5.7]{We95} is that the functor given by 
\[
V \longmapsto \Omega^\infty[(S^{\R^n \otimes V} \wedge \Theta)_{hO(n)}],
\] 
is $n$-polynomial, where $\Theta$ is a spectrum with an action of $O(n)$. This result hold true for unitary calculus, with a completely analogous proof. 

\begin{ex}\label{example: infinite loop polynomial}
Let $n$ be a non-negative integer. If $\Theta$ is a spectrum with an action of $\U(n)$, then the functor $F$ given by
\[
U \longmapsto \Omega^\infty[(S^{\C^n \otimes U} \wedge \Theta)_{h\U(n)}],
\]
is $n$-polynomial. 
\end{ex}
\begin{proof}
Since $F$ has a delooping, by Corollary \ref{cor: delooping of n-poly}, it suffices to show that $F^{(n+1)}$ is levelwise weakly contractible. From Proposition \ref{derivatives stable}, the $(n+1)$-th derivative is the homotopy fibre of 
\[
F^{(n)}(U) \longrightarrow \Omega^{2n}F^{(n)}(U \oplus \C).
\]
Iterating this process gives a sequence of derivatives in $\E_0$
\[
F^{(n)} \longrightarrow F^{(n-1)} \longrightarrow \dots \longrightarrow F^{(i)} \longrightarrow \dots \longrightarrow F^{(1)} \longrightarrow F.
\]
We follow Weiss \cite{We95}, and aim to identify this sequence with a sequence
\[
 F[n] \longrightarrow F[n-1] \longrightarrow \dots \longrightarrow F[i] \longrightarrow \dots \longrightarrow F[1] \longrightarrow F,
\]
where $F[i](U) = \Omega^\infty[(S^{nU} \wedge \Theta)_{h\U(n-i)}]$, and $\U(n-i)$ fixes the first $i$ coordinates. 

Each $F[i]$ comes with canonical inclusion maps
\[
\begin{split}
S^{iU} \wedge \Omega^\infty[(S^{nV} \wedge \Theta)_{h\U(n-i)}] &\hookrightarrow \Omega^\infty[S^{iU} \wedge (S^{nV} \wedge \Theta)_{h\U(n-i)}] \\
													&=  \Omega^\infty[(S^{iU} \wedge S^{nV} \wedge \Theta)_{h\U(n-i)}] \\
												&\hookrightarrow\Omega^\infty[(S^{nU} \wedge S^{nV} \wedge \Theta)_{h\U(n-i)}] \\
												&\rightarrow \Omega^\infty[(S^{n(U \oplus V)} \wedge \Theta)_{h\U(n-i)}] \\
\end{split}
\]
where the equality holds since $\U(n-i)$ fixes the first $i$-coordinates, hence fixes $\C^i$. This series of inclusions defines a structure map $\sigma : S^{iU} \wedge F[i](V) \longrightarrow F[i](U \oplus V)$, hence $F[i] \in \E_i$. 

Moreover $F[n]$ is an $n\Omega$-spectrum (see Definition \ref{def: n-omega spectra}) in $\E_n$, compare \cite[Example 2.3]{We95}. We show that $F[i+1]$ is levelwise weakly equivalent to $F[i]^{(1)}$, and since $F[n]$ is an $n\Omega$-spectrum, $F[n]^{(1)}$ will vanish, and hence so too will the $(n+1)$-st derivative of $F$. We inductively calculate $F[i]^{(1)}$.

\[
\begin{split}
F[i]^{(1)}(U) &= \hofibre \left[ F[i](U) \longrightarrow \Omega^{2i}F[i](\C \oplus U)  \right] \\
		    &=  \hofibre \left[ \Omega^\infty[(S^{nU} \wedge \Theta)_{h\U(n-i)}]  \longrightarrow \Omega^{2i}\Omega^\infty[(S^{n(\C \oplus U)} \wedge \Theta)_{h\U(n-i)}] \right] \\
		    &= \Omega^\infty \hofibre\left[  (S^{nU} \wedge \Theta)_{h\U(n-i)}  \longrightarrow \Omega^{2i}(S^{n(\C \oplus U)} \wedge \Theta)_{h\U(n-i)}\right] \\
	     &\simeq \Omega^\infty \hofibre\left[  (S^{nU} \wedge \Theta)_{h\U(n-i)}  \longrightarrow \Omega^{2i}(S^{2n} \wedge S^{nU} \wedge \Theta)_{h\U(n-i)}\right] \\
	     &\simeq \Omega^\infty \hofibre\left[  (S^{nU} \wedge \Theta)_{h\U(n-i)}  \longrightarrow \Omega^{2i}(S^{2i} \wedge S^{2(n-i)} \wedge S^{nU} \wedge \Theta)_{h\U(n-i)}\right] \\
	     &\simeq \Omega^\infty \hofibre\left[  (S^{nU} \wedge \Theta)_{h\U(n-i)}  \longrightarrow \Omega^{2i}\Sigma^{2i}(S^{2(n-i)} \wedge S^{nU} \wedge \Theta)_{h\U(n-i)}\right] \\
	     &\simeq \Omega^\infty \hofibre\left[  (S^{nU} \wedge \Theta)_{h\U(n-i)}  \longrightarrow (S^{2(n-i)} \wedge S^{nU} \wedge \Theta)_{h\U(n-i)}\right] \\.
\end{split}
\]
Now consider the map, $S^{nU} \wedge \Theta  \longrightarrow S^{2(n-i)} \wedge S^{nU} \wedge \Theta$. This map is 
\[
S^0 \wedge S^{nU} \wedge \Theta \xrightarrow{\iota \wedge \id \wedge \id} S^{2(n-i)} \wedge S^{nU} \wedge \Theta,
\]
where $\iota : S^0 \longrightarrow S^{2(n-i)}$ is the canonical inclusion. The map $\iota$ has stable homotopy fibre $S^{2(n-i)-1}$. Hence $\iota \wedge \id \wedge \id$ has homotopy fibre $S^{2(n-i)-1} \wedge S^{nU} \wedge \Theta$, where $\U(n-i)$ acts on $S^{2(n-i)-1}$ by identifying $S^{2(n-i)-1}$ with the unit sphere $S(\C^{n-i})$ of $\C^{n-i}$. Taking homotopy orbits commutes with homotopy fibres, so the map
\[
(S^{nU} \wedge \Theta)_{h\U(n-i)}  \longrightarrow (S^{2(n-i)} \wedge S^{nU} \wedge \Theta)_{h\U(n-i)}
\]
has homotopy fibre, 
\[
(S^{2(n-i)-1} \wedge S^{nU} \wedge \Theta)_{h\U(n-i)}.
\]
It follows that 
\[
F[i]^{(1)}(U) \simeq \Omega^\infty[(S^{2(n-i)-1} \wedge S^{nU} \wedge \Theta)_{h\U(n-i)}] \simeq \Omega^\infty[(S^{nU} \wedge \Theta)_{h\U(n-i-1)}]= F[i+1](U).
\]

The last weak equivalence follows since we may identify $S^{2n-1}$ with $\U(n)/\U(n-1)$, and
\[
\begin{split}
(S(\C^{n})_+ \wedge X)_{h\U(n)} &= E\U(n)_+ \wedge_{\U(n)}( \U(n)/\U(n-1)_+ \wedge X) = E\U(n)_+ \wedge_{\U(n-1)} X \\
&\simeq E\U(n-1)_+ \wedge_{\U(n-1)} X \simeq X_{h\U(n-1)}.\\
\end{split}
\]
\end{proof}

\subsection{The $n$-homogeneous model structure} The $n$-th derivative allows for us to equip $\E_0$ with a model structure which captures the homotopy theory of both polynomial and homogeneous functors of degree less than or equal $n$. This model structure is a right Bousfield localisation (or a cellularization) of the $n$-polynomial model structure.

\begin{prop} \label{n-homogeneous model structure}
Let $n$ be a non-negative integer. There is a topological model structure on $\E_0$ where the weak equivalences are those maps $f$ such that $\ind_0^nT_n f$ is a weak equivalence in $\E_0$, the fibrations are the fibrations of the $n$-polynomial model structure and the cofibrations are those maps with the left lifting property with respect to the acyclic fibrations. The fibrant objects are $n$-polynomial and the cofibrant-fibrant objects are the projectively cofibrant $n$-homogeneous functors.
\end{prop}
\begin{proof}
Right Bousfield localising $n\poly\E_0$ at the set of objects,
\[
\mathcal{K}_n = \{\mathrm{Th}(\gamma_n(U,-)) \ : \ U \in \J \},
\]
we achieve the stated model structure.
\end{proof}

We call this the $n$-homogeneous model structure and denote this model structure by $n\homog\E_0$. This is the unitary version of the model structure given by Barnes and Oman \cite[Proposition 6.9]{BO13}. We further characterise the $n$-homogeneous model structure in Section \ref{section: classification of homogeneous functors} once we have a classification of the $n$-homogeneous functors.

\section{An intermediate category}\label{section: intermediate category}

In \cite{We95}, Weiss gives a (zig-zag) equivalence up to homotopy between the categories of $n$-homogeneous functors and orthogonal spectra with an action of $O(n)$. Barnes and Oman give a more concrete description of this equivalence in \cite{BO13} via the use of model categories. They construct an intermediate model category which is the natural home for the $n$-th derivatives of orthogonal functors. They proceed to show that the intermediate category is both Quillen equivalent to the category of orthogonal spectra with an action of $O(n)$, and  the $n$-homogeneous model structure. Their use of model structures formalises the ``up to homotopy'' approach of Weiss. 

We start by giving the definition of the unitary intermediate category, $\U(n)\E_n$, and extending the restriction-induction adjunction into an adjunction between these intermediate categories. From this we construct Quillen equivalences analogous to those constructed by Barnes and Oman.

The category $\J_n$ is $\U(n)\T$-enriched via the induced action of $\U(n)$ on $\J_n(U,V)$ by the regular representation action of $\U(n)$ on $\C^n$. 

\begin{definition}
Let $n$ be a non-negative integer. The \textit{intermediate category} $U(n)\E_n$ is the category of $\U(n)\T$-enriched functors from $\J_n$ to $\U(n)\T$. 
\end{definition}

Define $n\mathbb{S} : \J_n \longrightarrow \T$ to be the functor given by $U \longmapsto S^{nU}$. Note that $n\mathbb{S}(U) = \J_n(0,U)$. Using the standard Day convolution product, one can verify that $n\mathbb{S}$ is a commutative monoid in $\I$-spaces, where $\I$ is the category of finite-dimensional complex inner product spaces with linear isometric isomorphisms. Moreover, this multiplication is $\U(n)$-equivariant and hence $n\mathbb{S}$ is also a commutative monoid in $\U(n)\I\T$, the category of $\U(n)$-equivariant $\I$-spaces.

\begin{prop}
Let $n$ be a non-negative integer. The category $\E_n$ is equivalent to the category of $n\mathbb{S}$-modules in $\I$-spaces, and the category $\U(n)\E_n$ is equivalent to the category of $n\mathbb{S}$-modules in $\U(n)$-equivariant $\I$-spaces. 
\end{prop}
\begin{proof}
By \cite{MMSS01}, the category of $n\mathbb{S}$-modules is a diagram category indexed on $\mathscr{J}_n$, where $\mathscr{J}_n$ is the category with the same objects as $\J_n$ and morphism spaces given by the enriched coend
\[
\mathscr{J}_n(V,W) \cong \int^{U \in \I} \I(U \oplus V, W) \wedge S^{nU}.
\]
It is then straightforward to check that the map specified by 
\[
\I(U \oplus V, W) \wedge S^{nU} \longrightarrow \J_n(V,W) ,\ \ (f, u) \longmapsto (f|_V, (\C^n \otimes f)(u)),
\]
is a suitably $\U(n)$-equivariant isomorphism. 
\end{proof}

\subsection{The $n$-stable model structure} This description as a category of modules allows for a stable model structure to be placed on $\U(n)\E_n$ similar to the stable model structure on spectra. For this we define the weak equivalences and fibrant objects. The constructions are as in \cite{MMSS01}. The model structure is a left Bousfield localisation of the projective model structure given below. 

\begin{lem}
Let $n$ be a non-negative integer. There is a cellular, proper, topological model structure on the category $\U(n)\E_n$ with weak equivalences and fibrations defined as levelwise weak homotopy equivalences and Serre fibrations respectively. The generating (acyclic) cofibrations are of the form
\[
\J_n(U,-) \wedge U(n)_+ \wedge i,
\]
for $U \in \J_n$ and $i$ a generating (acyclic) cofibration.
\end{lem}
\begin{proof}
This is essentially \cite[Theorem 6.5]{MMSS01} but the diagrams are in $U(n)\T$ rather than $\T$. 
\end{proof}

The weak equivalences are an alternation of the $\pi_*$-isomorphisms of spectra to take into account the structure maps $S^{2} \wedge X_k \longrightarrow X_{k+1}$. 

\begin{definition}
For $X \in \U(n)\E_n$ define the \textit{$n$-homotopy groups} of $X$ as
\[
n\pi_k X = \underset{q}{\colim}~ \pi_{2nq +k} X(\C^q).
\]
for all $k \in \Z$. A map $f: X \longrightarrow Y$ in $U(n)\E_n$ is a $n\pi_*$-isomorphism if it induces an isomorphism on all $n$-homotopy groups. 
\end{definition}

The fibrant objects will be similar to $\Omega$-spectra but take into account the structure maps of  $n\mathbb{S}$-modules.

\begin{definition}\label{def: n-omega spectra}
Let $n$ be a non-negative integer. An object $X \in \U(n)\E_n$ is an \textit{$n\Omega$-spectrum} if the adjoint structure maps 
\[
X(V) \longrightarrow \Omega^{nW}X(V \oplus W),
\] 
are weak homotopy equivalences for all $V,W \in \J_n$. 
\end{definition}

Denote the restricted composition map by
\[
\lambda_{U,V}^n : \J_n(U \oplus V, -) \wedge S^{nV} \longrightarrow \J_n(U,-).
\]
We factor $\lambda_{U,V}^n$ into a cofibration
\[
k_{U,V}^n : \J_n(U \oplus V, -) \wedge S^{nV} \longrightarrow M \lambda_{U,V}^n,
\]
where $M\lambda_{U,V}^n$ is the mapping cylinder of $\lambda_{U,V}^n$, and a deformation retraction 
\[
r_{U,V}^n :M\lambda_{U,V}^n \longrightarrow \J_n(U,-).
\]
Denote by $k_{U,V}^n \Box i$ the pushout product 
\[
k_{U,V}^n \Box i:  \J_n(U \oplus V, -) \wedge S^{nV} \wedge B) \cup_{\J_n(U \oplus V, -) \wedge S^{nV} \wedge A)} (M\lambda_{U,V}^n \wedge A) \longrightarrow M\lambda_{U,V}^n \wedge B,
\]
of the map $k_{U,V}^n : \J_n(U \oplus V, -) \wedge S^{nV} \longrightarrow M \lambda_{U,V}^n$ with $i: A \longrightarrow B$.

\begin{prop}[{\cite[Theorem 9.2]{MMSS01}}]\label{prop: n-stable model structure}
Let $n$ be a non-negative integer. There is a cofibrantly generated, proper, stable, topological model structure on the category $\U(n)\E_n$, where the weak equivalences are the $n\pi_*$-isomorphisms, the cofibrations are the q-cofibrations  and the fibrations are those maps which satisfy the right lifting property with respect to the acyclic q-cofibrations. The generating cofibrations are those of the projective model structure and the generating acyclic cofibrations are the union of the projective generating acyclic cofibrations together with the set  
\[
K_{V,W}^n \Box I_{\U(n)\T} := \{ k_{V,W}^n \Box i \ : \  i \in I_{\U(n)\T}, \ V, W \in \I\},
\]
where $I_{\U(n)\T}$ is the set of generating cofibrations for the underlying model structure on $\U(n)\T$.
\end{prop}

We call this the $n$-stable model structure. It is homotopically compactly generated by $\U(n)_+ \wedge n\mathbb{S}$.

The derivatives of $n$-polynomial objects are well behaved with respect to the $n$-stable model structure, in that they are $n\Omega$-spectra. Since $n\Omega$-spectra are the fibrant objects in the $n$-stable model structure, the following result indicates that the $n$-stable model structure gives homotopical control over the derivatives. The orthogonal version of this may be found in \cite[Proposition 5.12]{BO13} or \cite[Proposition 5.12]{We95}.

\begin{lem}\label{derivatives of n-poly are stable}
Let $n$ be a non-negative integer. If $E$ is an $n$-polynomial in $\E_0$, then for any $V \in \J_0$, the map 
\[
E^{(n)}(V) \longrightarrow \Omega^{2n} E^{(n)}(V \oplus \C),
\]
is a weak homotopy equivalence. 
\end{lem}

\section{The intermediate category as a category of spectra}\label{section: intermediate category as spectra}

The intermediate category, $\U(n)\E_n$ is Quillen equivalent to spectra with an action of  $\U(n)$. We digress from discussing the Taylor tower to discuss this equivalence. It will allow us to reduce the proof of the equivalence between the intermediate category and the $n$-homogeneous model structure to the consideration of spectra, see Theorem \ref{n-homog and intermediate QE}. We start with a discussion of unitary spectra, which is a version of the diagram spectra of Mandell et al. \cite{MMSS01}. 

\subsection{Unitary spectra} Let $\I$ be the category with the same objects as $\J$ and morphisms the complex linear isometric isomorphisms. The category of unitary spectra is the category of diagram spectra over $\I$, \cite[Definition 1.9]{MMSS01}. That is, a unitary spectrum $X$ is an $\mathbb{S}$-module in the category of $\I$-spaces, where $\mathbb{S}$ is the sphere functor which sends a complex inner product space to its one-point compactification. We will denote the category of unitary spectra as $\s^\mathscr{U}$. 

\begin{rem}\label{rem: def of unitary spectra}
Unitary spectra appear in the literature \cite[\S\S7.2]{Sc12} under the guise of a sequence of spaces $\{X_n\}_{n \in \N}$ such that $\U(n)$ acts on $X_n$, together with structure maps  
\[
X_n \wedge S^2 \longrightarrow X_{n+1},
\]
such that the iterated structure maps 
\[
X_n \wedge S^{2m} \longrightarrow X_{n+m},
\]
are $(\U(n) \times \U(m))$-equivariant. 
\end{rem}

Using the vector bundle construction of spectra, we can think of unitary spectra as the category of $\T$ enriched functors from $\J_1$ to $\T$, see \cite[Remark 2.7]{Sch19}, where Schwede gives an equivalence of categories between orthogonal $G$-spectra defined similar to unitary spectra in Remark \ref{rem: def of unitary spectra} and the definition of orthogonal $G$-spectra used by Mandell and May \cite{MM02}. 

Denote evaluation at $U\in \J_1$ by $\Ev_U : \s^\mathscr{U} \longrightarrow \T$, and its left adjoint by $F_U$. The category $\I$ satisfies all the required properties of a diagram category from \cite{MMSS01}. We can thus talk about diagram spectra indexed on $\I$ or diagram spaces indexed on $\J_1$. These categories are isomorphic by \cite[Theorem 2.2]{MMSS01}.

We now give it a stable model structure.  Similarly to \cite[Definition 8.4]{MMSS01} define a map 
\[
\lambda_{U,V} : F_{U\oplus V}S^V \longrightarrow F_{U}S^0,
\]
to be the map adjoint to the canonical inclusion 
\[
S^V \longrightarrow (F_{U}S^0)(U \oplus V) \simeq \U(U \oplus V)_+ \wedge_{\U(V)} S^W, w \longmapsto e \wedge w,
\]
where $e \in \U(U \oplus V)$ is the identity element, and the equivalence follows from the unitary version of \cite[Example 4.4]{MMSS01}.

\begin{definition}\label{definition of the set K}
Let $M\lambda_{U,V}$ be the mapping cylinder of $\lambda_{U,V}$. Then $\lambda_{U,V}$ factors as the composite of a q-cofibration 
\[
k_{U,V} : F_{U \oplus V}S^V \longrightarrow M\lambda_{U,V},
\]
and a deformation retraction $r_{U,V} : M\lambda_{U,V} \longrightarrow F_{U}S^0$. Let $K_{U,V}\Box I$ be the set of maps of the form $k_{U,V} \Box i$ for $i\in I$. Define $K$ to the union of the generating acyclic cofibrations of the projective model structure 
\[
FJ = \{F_Uj \ : \ j \in J,  \ U \in \I\},
\]
with the sets $K_{U,V} \Box I$.
\end{definition}

The $\pi_*$-isomorphisms for unitary spectra have to take into account the suspension coordinate $S^2$ of the structure maps. Hence for a unitary spectrum $X$ and $k \in \Z$, the $k$-th homotopy group of $X$ is defined as:
\[
\pi_kX = \underset{q}{\colim} \pi_{2q+k} X(\C^q).
\]

\begin{prop}[{\cite[Theorem 9.2]{MMSS01}}]
The category $\s^\mathscr{U}$ of unitary spectra is a cofibrantly generated stable model category with respect to the $\pi_*$-isomorphisms q-fibrations and q-cofibrations. The set of generating q-cofibrations is the set 
\[
FI = \{F_Ui \ : \ i \in I,  \ U \in \I\},
\]
and set of generating acyclic q-cofibrations is the set $K$ defined above.
\end{prop}

Let $\J^\mathscr{O}$ be the category of finite-dimensional real inner product subspaces of $\R^\infty$ with morphisms real linear isometries. There is a realification (decomplexification) functor $r: \J \longrightarrow \J^\mathscr{O}$ given by forgetting the complex structure, that is,  $r(\C^k) = \R^{2k}$. Precomposition with $r$ gives a functor $r^* : \s^\mathscr{O} \longrightarrow \s^\mathscr{U}$, which we call pre-realification. 

Pre-realification has a left adjoint $r_! : \s^\mathscr{U} \longrightarrow \s^\mathscr{O}$ given by the left Kan extension along $r$, that is, 
\[
r_!(X)(V) = \int^{U \in \J} \J_1^\mathscr{O}(r(U), V) \wedge X(U). 
\]

\begin{thm}[Theorem F]\label{thm: QE of models of spectra}
The adjoint pair 
\[
\adjunction{r_!}{\s^\mathscr{U}}{\s^\mathscr{O}}{r^*},
\]
is a Quillen equivalence.
\end{thm}
\begin{proof}
The right adjoint preserves acyclic fibrations (which are levelwise acyclic fibrations of based spaces, see \cite[Proposition 9.9]{MMSS01}) and fibrant objects. Moreover a standard cofinality argument shows that the right adjoint is homotopically conservative, that is, reflects weak equivalences. It is left to show that the derived unit of the adjunction is an isomorphism. Note that the left adjoint preserves coproducts. Then, since the stable model structure on unitary spectra is homotopically compactly generated by the unitary sphere spectrum $\mathbb{S}$ and both $\s^\mathscr{U}$ and $\s^\mathscr{O}$ are stable model categories, it suffices to show that the unit is an equivalence on the generator $\mathbb{S}$. Applying the left adjoint $r_!$ to the sphere $\mathbb{S}$, is equivalent to the functor $V \mapsto \mathbb{S}(c(V))$, where $c$ is the complexification functor, and the unit of the adjunction is equivalent to the map $\mathbb{S} \to r^*c^*\mathbb{S} = (c \circ r)^*\mathbb{S}$, which is readily seen to be a stable equivalence.
\end{proof}

\begin{cor}\label{cor: spectra QE}
Let $n$ be a non-negative integer. The adjoint pair 
\[
\adjunction{r_!}{\s^\mathscr{U}[\U(n)]}{\s^\mathscr{O}[\U(n)]}{r^*},
\]
is a Quillen equivalence.
\end{cor}

\begin{rem}
This relation between orthogonal and unitary spectra hints at a bigger relation between the two calculi. The pre-realification functor defines a functor between the orthogonal calculus input category and the unitary calculus input category. In future work we will examine how the pre-realification functor behaves with respect to $n$-polynomial and $n$-homogeneous functors. 
\end{rem}

\subsection{The Quillen equivalence}
The intermediate category is Quillen equivalent to the stable model structure on the category of orthogonal spectra with an action of $\U(n)$. 
To construct the required adjunction we first construct a $\U(n)\T$-enriched functor $\alpha_n: \J_n \longrightarrow \varepsilon^*\J_1$, where $\varepsilon^*\J_1$ is the $1$-st jet category equipped with the trivial $\U(n)$-action via the inclusion of the trivial subgroup $\varepsilon : \{e\} \to U(n)$. On objects, define $\alpha_n(V) = \C^n \otimes V = nV$. On morphisms, define $\alpha_n(f,x) = (\C^n \otimes f, x)$. This defines a $\U(n)\T$-enriched functor $\alpha_n : \J_n \longrightarrow \varepsilon^*\J_1$.

The adjunction is given as follows

\begin{prop}
Let $n$ be a non-negative integer. There is an adjoint pair 
\[
\adjunction{(\alpha_n)_!}{\U(n)\E_n}{\s^\mathscr{U}[\U(n)]}{(\alpha_n)^*},
\]
with $\alpha_n^* \Theta (V) = \Theta(nV)$, and $(\alpha_n)_!$ is the left Kan extension along $\alpha_n$. 
\end{prop}
\begin{proof}
Precomposition with $\alpha_n$ defines a functor 
\[
(\alpha_n)^* : \s^\mathscr{U}[U(n)] \longrightarrow \U(n)\E_n, 
\]
where $U(n)$-acts on $(\alpha_n)^*\Theta(V) = \Theta(nV)$ by composition of the internal $\U(n)$-action via the regular representation of $\U(n)$ on $\C^n \otimes V$ and the external $\U(n)$-action of $\Theta(nV) \in \U(n)\T$. By definition of $X$ these actions commute. Checking that this action gives a well defined object of $\U(n)\E_n$ follows as for the orthogonal version \cite[\S8]{BO13}. 

The left Kan extension along $\alpha_n$, denoted $(\alpha_n)_!$, may be written as the coend
\[
(\alpha_n)_!Y(V) = \int^{U \in \J_n} Y(U) \wedge \J_1(nU, V)
\]
which is suitably enriched since $\J_1$ acts on the left of $\J_1(nU,V)$ by the composition 
\[
\J_1(V,W) \wedge \J_1(nU, V) \longrightarrow \J_1(nU,W),
\]  
and $\J_n$ acts on the right via the composition
\begin{alignat*}{2}
 \J_1(nU, V) \wedge \J_n(W,U) &\longrightarrow \J_1(nU, V) \wedge \J_1(nW, nU) &&\longrightarrow \J_1(nW, V) \\
  ((g,y) ,(f,x)) &\longmapsto ((g,y),(\C^n\otimes f, x)) &&\longmapsto (g\circ (\C^n \otimes f), y + (\id_{\C^n} \otimes g)(x)).
\end{alignat*}
The adjunction then follows by a straightforward calculus of coends argument. 
\end{proof}

\begin{thm}[Theorem D]\label{QE: intermediate and spectra} 
Let $n$ be a non-negative integer. The  adjoint pair 
\[
\adjunction{(\alpha_n)_!}{\U(n)\E_n}{\s^\mathscr{U}[\U(n)]}{(\alpha_n)^*},
\]
is a Quillen equivalence.
\end{thm}
\begin{proof}
As it is defined by precomposition, the right adjoint preserves acyclic fibrations and fibrant objects. Suppose $f: \Theta \longrightarrow \Psi$ is a map of spectra such that $(\alpha_n)^*f: (\alpha_n)^*\Theta \longrightarrow (\alpha_n)^*\Psi$ is a $n\pi_*$-isomorphism. Then 
\[
\pi_k \Theta = \underset{q}{\colim}~ \pi_{2q+k} \Theta(\C^q) \cong \underset{q}{\colim}~ \pi_{2nq+k} \Theta(\C^{nq}) \cong \underset{q}{\colim}~ \pi_{2q+k} \Theta(n\C^q) = n\pi_k (\alpha_n)^*\Theta.
\]
A similar calculation shows that $\pi_k \Psi \cong n\pi_k (\alpha_n)^*\Psi$, and hence the right adjoint is homotopically conservative. 

Since both model categories are stable and the left adjoint commutes with coproducts, it suffices to show that the (derived) unit is an isomorphism on the homotopical compact generator $\U(n)_+ \wedge n\mathbb{S}$ of $U(n)\E_n$. This is similar to Theorem \ref{thm: QE of models of spectra}, where we write the left adjoint as a coend and work through the definitions. 
\end{proof}

\begin{cor}\label{cor: intermediate and ortho spectra}
Let $n$ be a non-negative integer. The  adjoint pair 
\[
\adjunction{(\alpha_n \circ r)_!}{\U(n)\E_n}{\s^\mathscr{O}[\U(n)]}{(\alpha_n \circ r)^*},
\]
is a Quillen equivalence.
\end{cor}
\begin{proof}
This is Theorem \ref{QE: intermediate and spectra} and Theorem \ref{thm: QE of models of spectra}, together with the fact that composition of left (resp. right) Quillen functors is a left (resp. right) Quillen functor. 
\end{proof}

\section{Differentiation as a Quillen functor}\label{section: diff as a QF}
We construct a Quillen equivalence between the intermediate category and the $n$-homogeneous model structure on $\E_0$, which allows for a slicker proof of the characterisation of $n$-homogeneous functors than the orthogonal version of Weiss \cite[Theorem 7.3]{We95}.

There is an adjunction 
\[
\adjunction{\res_0^n}{\E_n}{\E_0}{\ind_0^n},
\]
which we want to extend to an adjunction between $U(n)\E_n$ and $\E_0$. To do this, we combine the restriction-induction adjunction with the change of group adjunctions from \cite{MM02}.

\begin{definition}\label{def of restriction-orbit functor}
	Let $0\leq m <n$. Define the \textit{restriction-orbit} functor 
\[
\res_n^m/ \U(n-m): \U(n)\E_n \longrightarrow \U(m)\E_m
\] 
as $X \longrightarrow (X \circ i_m^n)/\U(n-m)$, where $ ((X \circ i_m^n)/\U(n-m))(V) = X(V)/\U(n-m).$ 
\end{definition}
 
The $\U(n-m)$-orbits functor $(-)/\U(n-m) : \U(n)\T \longrightarrow \U(m)\T$ has a right adjoint. Let $A$ be a $U(m)$-space, then $A$ is a $(\U(m)\times \U(n-m))$-space by letting $\U(n-m)$ act trivially. To distinguish we denote the $(\U(m)\times \U(n-m))$-space $A$ by $\varepsilon^*A$. Letting $F(-,-)$ denote the internal function object in $(\U(m)\times \U(n-m))$-spaces, we define the inflation of $A$ as
\[
\mathrm{CI}_m^n A = F_{\U(m)\times \U(n-m)}(\U(n)_+,  \varepsilon^*A)
\]
where $\U(m) \subset \U(n)$ acts on the first $m$-coordinates and $\U(n-m)$ acts on the latter $(n-m)$-coordinates. 

\begin{definition}\label{inflation-induction}
	Let $0\leq m <n$. Define the \textit{inflation-induction} functor $\ind_m^n \mathrm{CI} : \U(m)\E_m \longrightarrow \U(n)\E_n$ to be 
\[
(\ind_m^n \mathrm{CI} (X))(V) = \U(m)\E_m(\J_n(V,-), \mathrm{CI}_m^n X).
\]
\end{definition} 

In particular, there is an adjunction
\[
\adjunction{\res_0^n/\U(n)}{\U(n)\E_n}{\E_0}{\ind_0^n\varepsilon^*},
\]
where $\varepsilon$ is the inclusion of the trivial subgroup into $\U(n)$. This adjunction is a Quillen adjunction.

\begin{prop}\label{prop: QA for differentiation}
Let $n$ be a non-negative integer. The adjoint pair
\[
\adjunction{\res_0^n/\U(n)}{\U(n)\E_n}{n\homog\E_0}{\ind_0^n\varepsilon^*},
\]
is a Quillen adjunction when $\U(n)\E_n$ is equipped with the $n$-stable model structure.
\end{prop}
\begin{proof}
The projective model structure on $\U(n)\E_n$ is cofibrantly generated, hence by \cite[Lemma 2.1.20]{Ho99} it suffices to show that the left adjoint preserves the generating (acyclic) cofibrations. Restriction-orbits applied to a generating (acyclic) cofibration yeilds $\J_n(V,-) \wedge i$ where $i$ is a generating (acyclic) cofibration of $\T$. 

Since the projective model structure on $\E_0$ is cofibrantly generated, $\J_0(V,-)$ is cofibrant. The sphere bundle $S\gamma_n(V,-)_+$ is homeomorphic to the homotopy colimit
\[
\underset{0 \neq U \subset \C^{n+1}}{\hocolim}~ \J_0(U \oplus V, -),
\]
and hence $S\gamma_n(V,-)_+$ is cofibrant by \cite[Theorem 18.5.2(1)]{Hi03}. Since $ S\gamma_n(V,-)_+$ and  $\J_0(V,-)$ are both cofibrant, and the mapping cone of a map between cofibrant objects are cofibrant, it follows that $\J_{n+1}(V,-)$ is cofibrant, as the mapping cone of  $S\gamma_n(V,-)_+ \longrightarrow \J_0(V,-)$. Hence, $\J_n(V,-) \wedge i$ is an (acyclic) cofibration in $\E_0$, and there is a Quillen adjunction
\[
\adjunction{\res_0^n/\U(n)}{\U(n)\E_n}{\E_0}{\ind_0^n\varepsilon^*},
\]
when $\U(n)\E_n$ is equipped with the projective model structure. 

Composition of left (resp. right) Quillen functors is well behaved, see \cite[\S1.3.1]{Ho99}, hence the Quillen adjunction between projective model structures extends through the Quillen adjunction 
\[
\adjunction{\mathbbm{1}}{\E_0}{n\poly\E_0}{\mathbbm{1}},
\]
to a Quillen adjunction
\[
\adjunction{\res_0^n/\U(n)}{\U(n)\E_n}{n\poly\E_0}{\ind_0^n\varepsilon^*}.
\]
The $n$-stable model structure is a left Bousfield localisation of $\U(n)\E_n$, hence to extend to a Quillen adjunction on the $n$-stable model structure it suffices by \cite[Theorem 3.1.6, Proposition 3.3.18]{Hi03} to show that $\ind_0^n \varepsilon^*$ sends fibrant objects in $n\poly\E_0$ to fibrant objects in $\U(n)\E_n$. By Proposition \ref{prop: n-poly model structure} and Proposition \ref{prop: n-stable model structure}, this reduces to showing that inflation-induction sends an $n$-polynomial object to a $n\Omega$-spectrum. This is precisely the content of Lemma \ref{derivatives of n-poly are stable}.

We now extend to the $n$-homogeneous model structure - which is a right Bousfield localisation of $n\poly\E_0$ - using \cite[Proposition 3.3.18]{Hi03}. Suppose $f: E \longrightarrow F$ is a $\mathcal{K}_n$-cellular equivalence between fibrant ($n$-polynomial) objects, then
\[
\ind_0^n\varepsilon^* E(V) = \E_0(\J_n(V,-), E ) \longrightarrow \E_0(\J_n(V,-) , F) = \ind_0^n\varepsilon^*F(V),
\]
is a weak homotopy equivalence. It follows by definition that that $\ind_0^n \varepsilon^* f$ is a levelwise weak equivalence hence a $n$-stable equivalence. An application of \cite[Proposition 3.3.18]{Hi03} yields the result.  
\end{proof}

We now give the desired Quillen equivalence. We give a different proof to that of Barnes and Oman \cite[Theorem 10.1]{BO13} for the orthogonal calculus case. We begin with an example which will be useful. The proof of which may be found in \cite[Example 6.4]{We95}.

\begin{ex}\label{Weiss 6.4} 
Let $\Theta$ be a spectrum with $\U(n)$-action. The functors $E$ and $F$, given by the formulae
\[
\begin{split}
E(U) &= [\Omega^\infty(S^{nU} \wedge \Theta)]_{h\U(n)} \\
F(U) &= \Omega^\infty[(S^{nU} \wedge \Theta)_{h\U(n)}]
\end{split}
\]
are weakly equivalent in $n\poly\E_0$, that is $T_n E \longrightarrow T_nF$ is a levelwise weak equivalence. 
\end{ex}

Recall that we denote cofibrant replacement by $Q$ and fibrant replacement by $R$. The model structures in which we are carrying our the respective (co)fibrant replacements should be clear from context. 

\begin{thm}[Theorem E]\label{n-homog and intermediate QE}
Let $n$ be a non-negative integer. The adjoint pair
\[
\adjunction{\res_0^n/\U(n)}{\U(n)\E_n}{n\homog\E_0}{\ind_0^n\varepsilon^*},
\]
is a Quillen equivalence.
\end{thm}
\begin{proof}
We use \cite[Corollary 1.3.16]{Ho99}. The right adjoint reflects weak equivalences between fibrant objects. Indeed, let $f: E \longrightarrow F$ be a map in $\E_0$ such that $\ind_0^n\varepsilon^* f: \ind_0^n \varepsilon^* E \longrightarrow \ind_0^n \varepsilon^* F$ is an $n\pi_*$-isomorphism. Without loss in generality we may assume that both $E$ and $F$ are $n$-polynomial. It follows that 
\[
\ind_0^n\varepsilon^* T_n f: \ind_0^n \varepsilon^* T_nE \longrightarrow \ind_0^n \varepsilon^* T_n F,
\]
is a levelwise weak equivalence. 

It is left to show that the derived unit of the adjunction is an equivalence on cofibrant objects $X \in \U(n)\E_n$. Denote by $\Theta$ the derived image of $X$ under the Quillen equivalence between the intermediate category and the category of spectra with an action of $\U(n)$, i.e., $\Theta = (\alpha_n \circ r)_!X$. It follows that $X$ is $n$-stably equivalent to $(\alpha_n \circ r)^\ast\fibrep \Theta$. It suffices to show that the derived until is an equivalence on objects of $\U(n)\E_n$ of the form $(\alpha_n \circ r)^\ast \Theta$ with $\Theta$ a bifibrant spectrum with an action of $\U(n)$. The derived unit map in question is given by is given by 
\[
\cofrep(\alpha_n \circ r)^\ast \Theta \longrightarrow \ind_0^n \varepsilon^* T_n \res_0^n (\cofrep(\alpha_n \circ r)^\ast \Theta)/\U(n).
\]

The codomain of the unit map is levelwise weakly equivalent to the functor
\[
\ind_0^n\varepsilon^\ast T_n (E\U(n)_+ \wedge_{\U(n)} \res_0^n ((\alpha_n \circ r)^\ast \Theta)/\U(n),
\]
since by the unitary calculus version of \cite[Lemma 9.3]{BO13}, we can up to levelwise weak equivalence identify the left derived functor or the restriction-orbits functor with the functor
\[
E\U(n)_+ \wedge_{\U(n)} \res_0^n.
\]

At $U \in \J$, there is an identification
\[
(\alpha_n \circ r)^*\Theta(U) = \Theta (r(nU)) \simeq \Omega^\infty(S^{nU} \wedge \Theta),
\]
hence $\res_0^n(\cofrep(\alpha_n \circ r)^* \Theta)/\U(n)$ is levelwise weakly equivalent to the functor given by 
\[
U \longmapsto E\U(n)_+ \wedge_{\U(n)}[\Omega^\infty(S^{nU} \wedge \Theta)] = [\Omega^\infty (S^{nU} \wedge \Theta)]_{h\U(n)}.
\]
By Example \ref{Weiss 6.4} this last is $T_n$-equivalent to the functor given by 
\[
U \longmapsto [\Omega^\infty(S^{nU} \wedge \Theta)_{h\U(n)}].
\]
Example \ref{example: infinite loop polynomial} calculates the $n$-th derivative of this functor as the functor given by 
\[
U \longmapsto \Omega^\infty(S^{nU} \wedge \Theta),
\]
which in turn is levelwise weakly equivalent to $(\alpha_n \circ r)^* \Theta$, and the result follows. 
\end{proof}

\section{The classification of homogeneous functors}\label{section: classification of homogeneous functors}

We return now to discussing the Taylor tower. We have seen that the layers of the tower are homogeneous. We now show that $n$-homogeneous functors (in particular the layers of the tower) are completely determined by spectra with an action of $\U(n)$. This result is an extension of \cite[Theorem 7.3]{We95}. Here, Weiss shows that an orthogonal $n$-homogeneous functor $E$ is levelwise weakly equivalent to the functor 
\[
V \longmapsto \Omega^\infty[(S^{\R^n \otimes V} \wedge \Theta_E^n)_{hO(n)}].
\]

\subsection{Classification of homogeneous functors} The above Quillen equivalence gives an equivalence between the homotopy category of the $n$-homogeneous model structure and the homotopy category of the $n$-stable model structure. It follows that for an object $F \in \E_0$, inflation-induction and the left adjoint to $(\alpha_n \circ r)^*$ determine a spectrum $\Psi_F^n$ with an action of $\U(n)$. That is, $\Psi_F^n = (\alpha_n \circ r)_!\ind_0^n\varepsilon^*F$.

With this we can classify $n$-homogeneous functors, analogous to Weiss \cite[Theorem 7.3]{We95} and Barnes-Oman \cite[Theorem 10.3]{BO13}. The proof here - aided by the language of model categories and localisations - is significantly more straightforward that than of \cite[Theorem 7.3]{We95}, yet still rather technical. 

\begin{thm}[Theorem A]\label{characterisation of homogeneous functors}
Let $n$ be a non-negative integer. Let $F \in \E_0$ be $n$-homogeneous for some $n >0$. Then $F$ is levelwise weakly equivalent to the functor defined as 
\begin{equation*}\label{char of homog functors}
U \longmapsto \Omega^\infty [(S^{nU} \wedge \Psi_F^n)_{h\U(n)}].
\end{equation*}
\end{thm}
\begin{proof}
Let $F$ be $n$-homogeneous and define a new functor $E$ to be $E(U) = (\ind_0^n\varepsilon^* F(U))_{h\U(n)}$. This functor is $T_n$-equivalent to the functor $G$ defined as
\[
G(U) = \Omega^\infty[(S^{nU} \wedge \Psi_F^n)_{h\U(n)}],
\]
which is $n$-polynomial by Example \ref{example: infinite loop polynomial}. Indeed, the $T_n$-equivalence follows since 
\[
E(U) = (\ind_0^n \varepsilon^* F(U))_{h\U(n)} = (\Psi_F^n (r(nU)))_{h\U(n)} \simeq  [\Omega^\infty(S^{nU} \wedge \Psi_F^n)]_{h\U(n)}.
\]
By Example \ref{Weiss 6.4}, the functor $E$ (through the above equivalence) is in turn $T_n$-equivalent to the functor $G$
\[
U \longmapsto \Omega^\infty[(S^{nU} \wedge \Psi_F^n)_{h\U(n)}].
\] 
Since $G$ is levelwise weakly equivalent to $T_nG$, it follows that $T_nE$ is levelwise weakly equivalent to $G$.

Since $\ind_0^n \varepsilon^*$ is a right Quillen functor, $\ind_0^n \varepsilon^* T_nE$ is levelwise weakly equivalent to $\ind_0^n \varepsilon^* G$. By Example \ref{example: infinite loop polynomial}, we may identify the $n$-th derivative of $G$, $\ind_0^n \varepsilon^* G$, with the functor $G[n]$, given by
\[
U \longmapsto \Omega^\infty (S^{nU} \wedge \Psi_F^n).
\]

This last is levelwise weakly equivalent to $\ind_0^n \varepsilon^* F$, since the above functor is $n$-stably equivalent to $\ind_0^n \varepsilon^* F$ and both are fibrant in $\U(n)\E_n$.

Since $G$ is levelwise weakly equivalent to $T_nE$, it follows that $\ind_0^n\varepsilon^* T_nE$ is levelwise weakly equivalent to $\ind_0^n\varepsilon^* T_nF$. A double application of Whiteheads' Theorem \cite[Theorem 3.2.13]{Hi03},  for right and left Bousfield localisations respectively yields the result.
\end{proof}


\subsection{Characterising the $n$-homogeneous model structure} With this classification result we can further characterise the weak equivalences and cofibrations of the $n$-homogeneous model structure. This new characterisation is also true in the orthogonal calculus setting. These results are similar to those for Goodwillie calculus \cite{BR14} since the construction of the model categories are similar. However there are substantial differences, for example the (homotopy) cross effect functor plays an important role in Goodwillie calculus, and the model structures of \cite{BR14}, but has no natural analogue in our theory. 

We start with the weak equivalences. The construction of the derivative is quite complex, hence detecting weak equivalences via  $\ind_0^nT_n$ can be laborious. We show, using the classification theorem for $n$-homogeneous functors, Theorem \ref{characterisation of homogeneous functors}, that weak equivalences are detected by $D_n$, where 
\[
D_nF = \hofibre[T_nF \longrightarrow T_{n-1}F].
\]

\begin{prop}\label{char of n-homog equivs}
Let $n$ be a non-negative integer. A map $f: E \longrightarrow F$ is an $\ind_0^nT_n$-equivalence if and only if $D_nf: D_nE \longrightarrow D_nF$ is a levelwise weak equivalence.
\end{prop}
\begin{proof}
Suppose that $D_nf$ is a levelwise weak equivalence. Since $\ind_0^n$ preserves levelwise weak equivalences, $\ind_0^n D_nf$ is a levelwise weak equivalence. Moreover as the $n$-th derivative of an $(n-1)$-polynomial object is levelwise weakly contractible, $\ind_0^n T_{n-1}E$ and $\ind_0^n T_{n-1}F$ are both levelwise weakly contractible and hence levelwise weakly equivalent. The following diagram made from the homotopy fibre sequences defining $D_nE$ and $D_nF$ together with an application of the Five Lemma implies that $\ind_0^nT_n f$ is a levelwise weak equivalence.
\[
\xymatrix{
\ind_0^n D_n E \ar[r] \ar[d]_{\ind_0^n D_n f} & \ind_0^n T_n E \ar[r] \ar[d]^{\ind_0^nT_nf} & \ind_0^nT_{n-1}E \ar[d]^{\ind_0^n T_{n-1}f} \\
\ind_0^n D_n F \ar[r] & \ind_0^n T_n F  \ar[r] & \ind_0^nT_{n-1}F \\
}
\]

Conversely suppose that $\ind_0^n T_n f$ is a levelwise weak equivalence. Then the spectra $\Psi_E^n$ and $\Psi_F^n$ are stably equivalent. As such the spectra  $(S^{nV} \wedge \Psi_E^n)_{h\U(n)}$ and $(S^{nV} \wedge \Psi_F^n)_{h\U(n)}$ are stably equivalent for every $V \in \J$. Since for any spectrum $\Theta$, $\pi_n\Theta = \pi_n \Omega^\infty \Theta$, we have that $\Omega^\infty[(S^{nV} \wedge \Psi_E^n)_{h\U(n)}]$ and $\Omega^\infty[(S^{nV} \wedge \Psi_F^n)_{h\U(n)}]$ are weakly homotopy equivalent. By the classification of $n$-homogeneous functors it follows that $D_nE$ and $D_nF$ are levelwise weakly equivalent. 
\end{proof}

It is also possible to characterise the acyclic fibrations.

\begin{prop}\label{characterisation of acyclic fibs}
Let $n$ be a non-negative integer. A map $f: E \longrightarrow F$ is an acyclic fibration in the $n$-homogeneous model structure if and only if it is a fibration in the $(n-1)$-polynomial model structure and an $\ind_0^nT_n$-equivalence.
\end{prop}
\begin{proof}
A map is a fibration in the $(n-1)$-polynomial model structure if and only if it is a levelwise fibration and the square 
\[
\xymatrix{
E \ar[r] \ar[d]_f & T_{n-1} E \ar[d]^{T_{n-1}f} \\
F \ar[r] & T_{n-1}F\\
}
\]
is a homotopy pullback square. 
Consider the diagram
\[
\xymatrix{
E \ar[r] \ar[d]_f & T_{n} E \ar[d]^{T_{n}f} \ar[r] & T_{n-1} E \ar[d]^{T_{n-1}f}\\
F \ar[r] & T_{n}F \ar[r] & T_{n-1}F. \\
}
\]
If $f: E \longrightarrow F$ is an acyclic fibration in the $n$-homogeneous model structure, then it is an $\ind_0^nT_n$-equivalence and a fibration in in $n$-polynomial model structure, since $n\homog\E_0$ is a right Bousfield localisation of $n\poly\E_0$. It follows that the left hand square is a homotopy pullback. The right hand square in the above diagram is also homotopy pullback since the horizontal (homotopy) fibres are levelwise equivalent, see \cite[Proposition 3.3.18]{MV15}. By the Pasting Lemma \cite[Proposition 13.3.15]{Hi03}, the outer rectangle is a homotopy pullback and hence $f$ is a fibration in the $(n-1)$-polynomial model structure. 

Conversely if $f$ is a fibration in the $(n-1)$-polynomial model structure then the outer rectangle in the above diagram is a homotopy pullback. Moreover, $f$ being an $\ind_0^nT_n$-equivalence yields the right hand square as a homotopy pullback, again by \cite[Proposition 3.3.18]{MV15}. Another application of the Pasting Lemma yields that the left hand square is a homotopy pullback and hence $f$ is a fibration in the $n$-polynomial model structure, and hence a fibration in the $n$-homogeneous model structure.
\end{proof}

This allows us to characterise the acyclic fibrations between fibrant objects. 

\begin{cor}\label{lem: acyclic fibrations in n-homog}
Let $n$ be a non-negative integer. A map $f:E \longrightarrow F$ between $n$-polynomial objects is an acyclic fibration in the $n$-homogeneous model structure if and only if it is a fibration in the $(n-1)$-polynomial model structure.
\end{cor}
\begin{proof}
By Proposition \ref{characterisation of acyclic fibs} it suffices to show that a fibration in $(n-1)\poly\E_0$ between $n$-polynomial objects is an $n$-homogeneous equivalence. Let $f: E \longrightarrow F$ be a fibration in $(n-1)\poly\E_0$ and $E$ and $F$, $n$-polynomial. Then we have a diagram 
\[
\xymatrix{
E \ar[r] \ar[d]_f & T_{n} E \ar[d]^{T_{n}f} \ar[r] & T_{n-1} E \ar[d]^{T_{n-1}f}\\
F \ar[r] & T_{n}F \ar[r] & T_{n-1}F. \\
}
\]
as in Proposition \ref{characterisation of acyclic fibs}. The outer square of this diagram is a homotopy pullback since $f: E \longrightarrow F$ is a fibration in $(n-1)\poly\E_0$. Since $E$ and $F$ are $n$-polynomial, the left-hand horizontal maps are levelwise equivalences and the right-hand square is a homotopy pullback. The result then follows by \cite[Proposition 3.3.18]{MV15} and Proposition \ref{char of n-homog equivs}.
\end{proof}

We now turn our attention to the cofibrations.

\begin{lem}\label{lem: cofibrations of n-homog}
Let $n$ be a non-negative integer. A map $f: X \longrightarrow Y$ is a cofibration in the $n$-homogeneous model structure if and only if it is a projective cofibration and an $(n-1)$-polynomial equivalence. 
\end{lem}
\begin{proof}
By definition, $f: X \longrightarrow Y$ is a $n$-homogeneous cofibration if and only if it has the left lifting property with respect to $n$-homogeneous acyclic fibrations. By right properness of $n\homog\E_0$ and \cite[Proposition 13.2.1]{Hi03}, this is equivalent to $f: X \longrightarrow Y$ having the left lifting property with respect to acyclic fibrations between fibrant objects. By Lemma \ref{lem: acyclic fibrations in n-homog} this is equivalent to having the left lifting property with respect to fibrations in the $(n-1)$-polynomial model structure. It follows that $f: X \longrightarrow Y$ is a cofibration in $n\homog\E_0$ if and only if it is an acyclic cofibration in $(n-1)\poly\E_0$, that is, if and only if it is a projective cofibration and an $(n-1)$-polynomial equivalence. 
\end{proof}

\begin{cor}\label{cofibrant objects of n-homog}
Let $n$ be a non-negative integer. The cofibrant objects of the $n$-homogeneous model structure are precisely those $n$-reduced projectively cofibrant objects.
\end{cor}
\begin{proof}
Let $E$ be cofibrant in $n\homog\E_0$ and apply Lemma \ref{lem: cofibrations of n-homog} to the map $\ast \longrightarrow E$. 
\end{proof}

\subsection{The complete Taylor tower}
The description of the $n$-homogeneous functors in particular gives a description of the layers of the Taylor tower. Hence at $U \in \J_0$, the Taylor tower of $F\in\E_0$ is
\[
\xymatrix@R-.5cm@C+2cm{
   & \vdots \ar[d]^{r_{n+1}} &\\
   & T_n F(U) \ar[d]^{r_n} &  \ar[l] \Omega^\infty[(S^{nU} \wedge \Psi_F^n)_{h\U(n)}] \\
   & T_{n-1}F(U) \ar[d]^{r_{n-1}}  &  \ar[l] \Omega^\infty[(S^{(n-1)U} \wedge \Psi_F^{n-1})_{h\U(n-1)}]\\
   & \vdots  \ar[d]^{r_2} &  \\ 
   &T_1 F(U) \ar[d]^{r_{1}}  &  \ar[l] \Omega^\infty[(S^U \wedge \Psi_F^1)_{h\U(1)}]\\
F(U) \ar@/^2pc/[uuuur]  \ar@/^1pc/[uuur]   \ar@/^/[ur] \ar[r] & F(\C^\infty).  &\\
}
\]

\section{Analyticity and convergence}\label{section: convergence}

One of the big questions in any version of functor calculus is that of convergence of the Taylor tower. This question is really two-fold; does the tower converge? And, if so, what is does the tower converge to? 

\begin{definition}
The Taylor tower of a functor $F \in \E_0$ \textit{converges at $U \in \J_0$} if the induced map 
\[
F(U) \longrightarrow \underset{n}{\holim} T_nF(U)
\]
is a weak homotopy equivalence. We say that $F$ is \textit{weakly $\rho$-analytic} if its Taylor tower converges at $U$ with $\dim(U) \geq \rho$. 
\end{definition}

Before we talk about special classes of functors for which the Taylor tower is known to converge, there is a classical approach. Since we have a homotopy fibre sequence 
\[
\Omega^\infty[(S^{nU} \wedge \Psi_F^n)_{h\U(n)}] \longrightarrow T_nF(U) \longrightarrow T_{n-1}F(U)
\]
 for all $U \in \J_0$, we can apply \cite[Section IX.4]{BK72}.

\begin{definition}
Let $F \in \E_0$. The \textit{Weiss spectral sequence} associated to F at $U \in \J_0$ is the homotopy spectral sequence of the tower of pointed spaces $\{T_nF(U)\}_{n \in \mathbb{N}}$ with $E_1$-page
\[
E_{s,t}^1 \cong \pi_{t-s} \Omega^\infty[(S^{sU} \wedge \Psi_F^s)_{h\U(s)}] \cong \pi_{t-s}(S^{sU} \wedge \Psi_F^s)_{h\U(s)},
\]
and abuts to
\[
\pi_* \underset{n}{\holim} T_nF(U).
\]
\end{definition}

This spectral sequence indicates why studying the layers of the tower is important, a firm grasp of the layers gives a firm grasp on the spectral sequence and hence on its limit. 

\begin{rem}
Bousfield and Kan \cite{BK72} provide a method of comparison between this spectral sequence for different towers of fibrations. It would be interesting to see how this Bousfield-Kan map of spectral sequences interacts with the comparison functors between orthogonal and unitary calculi. 
\end{rem}

\subsection{Agreement} We start with the notion of agreement to order $n$. This is the unitary version of order $n$ agreement from Goodwillie calculus \cite[Definition 1.2]{Go03}. These connectivity estimates played a crucial role in both Goodwillie calculus \cite[Proposition 1.17]{Go90} and orthogonal calculus \cite{We98}. 

\begin{definition}\label{unitary agreement}
Let $n$ be a non-negative integer. A map $p: F \longrightarrow G$ in $\E_0$ is \textit{an order $n$ unitary agreement} if there is some $\rho \in \mathbb{N}$ and $b \in \Z$ such that $p_U: F(U) \longrightarrow G(U)$ is $(2(n+1)\dim(U)-b)$-connected for all $U \in \J_0$, satisfying $\dim(U) \geq \rho$. We will say that \textit{$F$ agrees with $G$ to order $n$} if there is an order $n$ unitary agreement $p: F \longrightarrow G$ between them.
\end{definition}

When two functors agree to a given order, their Taylor tower agree to a prescribed level. The first result in that direction is the unitary analogue of \cite[Lemma e.3]{We98}.

\begin{lem}\label{connected argument}
Let $n$ be a non-negative integer and let $p : G \longrightarrow F$ be a map in $\E_0$. Suppose that there is $b \in \Z$ such that $p_U: G(U) \longrightarrow F(U)$ is $(2(n+1)\dim(U) - b)$-connected for all $U \in \J_0$. Then
\[
\tau_n(p)_U : \tau_n (G(U)) \longrightarrow \tau_n(F(U))
\] 
is $(2(n+1)\dim(U) -b +1)$-connected for all $U \in \J_0$. 
\end{lem}

Iterating this result, gives the following. 

\begin{lem}\label{agreement gives agreeing polynomials for unitary}
Let $n$ be a non-negative integer. If $p : F \longrightarrow G$ is an order $n$ unitary agreement, then $T_k F \longrightarrow T_k G$ is a levelwise weak equivalence for $k \leq n$.
\end{lem}
\begin{proof}
If $p: F \longrightarrow G$ is an order $n$ unitary agreement, then by Lemma \ref{connected argument},
\[
\tau_n(p)_U: \tau_n F(U) \longrightarrow \tau_n G(U)
\]
is $(2(n+1)\dim(U) -b+1)$-connected. Repeated application of Lemma \ref{connected argument} yields the result for $k=n$ since $T_kF(U) = \hocolim_i \tau_k^i F(U)$. The result follows for $k<n$ since if a map $f: X \longrightarrow Y$ is $k$-connected, then it is $(k-1)$-connected. 
\end{proof}

\begin{example}\label{nS n-reduced}
The functor $n\mathbb{S} : \J_0 \longrightarrow \T$ is $k$-reduced for all $k \leq n$. Indeed, the map $\ast \longrightarrow S^{nU}$ is $(2n\dim (U) -1)$-connected for all $U \in \J_0$, hence by Lemma \ref{connected argument}, $\tau_{n-1}(\ast) \longrightarrow \tau_{n-1}(n\mathbb{S}(U))$ is $(2n\dim(U))$-connected for all $V \in \J$. It follows from Lemma \ref{agreement gives agreeing polynomials for unitary} that $\ast \simeq T_{n-1}(\ast) \longrightarrow (T_{n-1}n\mathbb{S})(U)$ is a weak homotopy equivalence. The result for $k \leq n-1$ follows since a $k$-connected map is $k-1$-connected. 
\end{example}

The connectivity estimate of Lemma \ref{connected argument} gives conditions for the Taylor tower to be trivial to a prescribed level. 

\begin{lem}\label{lemma: connected level gives trivial tower}
Let $n$ be a non-negative integer. If $F \in \E_0$ is such that $F(U)$ has connectivity $(2(n+1)\dim U -b)$ for some constant $b$, then the Taylor tower of $F(U)$ is trivial up to and including level $n$. 
\end{lem}
\begin{proof}
If $F(U)$ has such a connectivity, then the map $\ast \longrightarrow F(U)$ is $(2(n+1)\dim(U) -b)$-connected. An application of Lemma \ref{agreement gives agreeing polynomials for unitary} yields that $\ast \longrightarrow T_k F(U)$ is a weak homotopy equivalence for all $k \leq n$. 
\end{proof}

\begin{ex}
The $n$-sphere functor $n\mathbb{S} : U \longmapsto S^{nU}$ satisfied the conditions of Lemma \ref{lemma: connected level gives trivial tower}, as $S^{nU}$ is $(2n\dim(U) -1)$-connected. Hence the first non-trivial polynomial approximation to $n\mathbb{S}$ is the $n$-th approximation $T_nn\mathbb{S}$. 
\end{ex}

Agreement with the $n$-polynomial approximation functor for all $n\geq 0$ gives convergence of the Taylor tower. 

\begin{lem}
If for all $n\geq 0$, a unitary functor $F$ agrees with $T_nF$ to order $n$ then the Taylor tower associated to $F$ converges to $F(U)$ at $U$ with $\dim(U) \geq \rho$.
\end{lem}
\begin{proof}
Since $F$ agrees with $T_nF$ to order $n$ for all $n$, the map $\eta : F(U) \longrightarrow T_nF(U)$ is $(2(n+1)\dim(U) -b)$-connected for all $n$ and all $U$ with $\dim U \geq \rho$. It follows that the map $F(U) \longrightarrow \holim T_nF(U)$ is a weak homotopy equivalence. 
\end{proof}

\subsection{Weakly Polynomial} The class of functors with this property are called weakly polynomial. They are an important class of functors in that they are weakly analytic, but defined in simpler terms. In particular weakly polynomial functors are more tractable for computations.  

\begin{definition}\label{def: weak poly}
Let $n$ be a non-negative integer. A unitary functor $F$ is \textit{weakly $(\rho,n)$-polynomial} if the map $\eta : F(U) \longrightarrow T_nF(U)$ is an agreement of order $n$ whenever $\dim(U) \geq \rho$. A functor is \textit{weakly polynomial} if it is weakly $(\rho,n)$-polynomial for all $n\geq 0$. 
\end{definition}

We start with a few examples. 

\begin{ex}\label{example: sphere functor analytic}
The sphere functor $\mathbb{S}: U \longmapsto S^U$ is weakly polynomial.
\end{ex}
\begin{proof}
The identity functor $\mathrm{Id}: \T \longrightarrow \T$ is $1$-analytic in the Goodwillie sense, \cite[Example 4.3]{Go91}. Barnes and Eldred \cite[Example 3.7]{BE16} shows that $S = S^*\mathrm{Id}$ is weakly $2$-polynomial, specifically, that the map $S^*\mathrm{Id}(V) \longrightarrow T_n S^*\mathrm{Id}(V)$ is an agreement of order $n$ for all $n\geq0$, where $S: V \longmapsto S^V$ is the orthogonal sphere functor. Their proof may be extended to the unitary calculus case giving that $\mathbb{S} = \mathbb{S}^*\mathrm{Id}$ is weakly polynomial for $\dim(V) \geq 1$.
\end{proof}

\begin{ex}\label{example: shifted spheres analytic}
Fix a constant $k \geq 0$. Then the functor given by $V \longmapsto S^{U+2k}$ is weakly polynomial. This follows from Example \ref{example: sphere functor analytic} since $V \longmapsto S^{U +2k}$ is equivalent to $\mathbb{S}(U\oplus \C^k)$. 
\end{ex}

The following result is an alteration of \cite[Theorem 4.1]{BE16} for weakly polynomial functors.

\begin{thm}[Theorem G]\label{quasi thm} 
Let $E, F\in \E_0$ are such that there is a homotopy fibre sequence 
\[
\xymatrix{
E(U) \ar[r] & F(U) \ar[r] &  F(U \oplus V)
}
\]
for $U,V \in \J$. Then 
\begin{enumerate}
\item If $F$ is weakly $(\rho,n)$-polynomial, then $E$ is weakly $(\rho,n)$-polynomial; and
\item If $E$ is weakly $(\rho,n)$-polynomial and $F(U)$ is $1$-connected whenever $\dim(U) \geq \rho$, then $F$ is weakly $(\rho, n)$-polynomial.
\end{enumerate}
\end{thm}
\begin{proof}
For $(1)$, simply note that $T_n$ preserves fibre sequence, hence there is a commutative diagram 
\[
\xymatrix{
E(U) \ar[r] \ar[d] & \ar[r] \ar[d] F(U) & F(U \oplus V) \ar[d] \\
T_nE(U) \ar[r] & \ar[r] T_nF(U) & T_nF(U \oplus V). \\
}
\]
It follows that since the middle and right hand vertical maps are agreements of order $n$, that the left hand vertical map is also an agreement of order $n$. 

For $(2)$ it suffices to consider a fibre sequence
\[
\xymatrix{
E(\C^\rho) \ar[r] & F(\C^\rho) \ar[r] & F(\C^{\rho+1}).
}
\]
We achieve the same map of fibre sequences, 
\[
\xymatrix{
E(\C^\rho) \ar[r] \ar[d] & \ar[r] \ar[d] F(\C^\rho) & F(\C^{\rho+1}) \ar[d] \\
T_nE(\C^\rho) \ar[r] & \ar[r] T_nF(\C^\rho) & T_nF(\C^{\rho+1}) \\
}
\]
in which the left hand vertical map is $(2(n+1)\rho -b)$-connected. Taking vertical fibres on the above diagram gives a diagram, 
\[
\xymatrix{
\mathcal{F}_{E(\C^\rho)} \ar[r] \ar[d] & \mathcal{F}_{F(\C^\rho)} \ar[r] \ar[d] & \mathcal{F}_{F(\C^{\rho+1})} \ar[d] \\
E(\C^\rho) \ar[r] \ar[d] & \ar[r] \ar[d] F(\C^\rho) & F(\C^{\rho+1}) \ar[d] \\
T_nE(\C^\rho) \ar[r] & \ar[r] T_nF(\C^\rho) & T_nF(\C^{\rho+1}). \\
}
\]
Since the lower left hand vertical map is $(2(n+1)\rho -b)$-connected, $\mathcal{F}_{E(\C^\rho)}$ is $(2(n+1)\rho -b-1)$-connected, hence the top right hand horizontal map is $(2(n+1)\rho -b)$-connected. It follows by \cite[Proposition 3.3.18]{MV15} and the fact that $F(\C^\rho)$ and $F(\C^{\rho+1})$ are path connected, that the lower right hand square is $(2(n+1)\rho -b)$-cartesian. 

Applying \cite[Proposition 3.3.20]{MV15} to the sequence of $(2(n+1)\rho -b)$-cartesian squares,
\[
\xymatrix{
 \ar[r] \ar[d] F(\C^\rho) & F(\C^{\rho+1}) \ar[d] \ar[r] & F(\C^{\rho+2}) \ar[d] \ar[r] & \dots \\
 \ar[r] T_nF(\C^\rho) & T_nF(\C^{\rho+1}) \ar[r] & F(\C^{\rho+2}) \ar[r] & \dots.\\
}
\]
yields for all $q \geq \rho$ and all $k \geq1$ a $(2(n+1)\rho -b)$-cartesian square
\[
\xymatrix{
F(\C^q) \ar[r] \ar[d] & F(\C^{q+k}) \ar[d] \\
T_nF(\C^q) \ar[r] & T_nF(\C^{q+k})\\
}
\]
Filtered homotopy colimits preserve $(2(n+1)\rho -b)$-cartesian squares, hence the square 
\[
\xymatrix{
F(\C^q) \ar[r] \ar[d] & F(\C^\infty) \ar[d] \\
T_nF(\C^q) \ar[r] & T_nF(\C^\infty) \\
}
\]
is $(2(n+1)\rho -b)$-cartesian. The unitary version of \cite[Lemma 5.14]{We95}, which holds since finite homotopy limits commute with filtered homotopy colimits, gives that $F(\C^\infty) \longrightarrow T_nF(\C^\infty)$ is a weak equivalence, in particular, it is $(2(n+1)\rho -b)$-connected. It follows by \cite[Proposition 3.3.11]{MV15} that the map 
\[
F(\C^q) \longrightarrow T_nF(\C^q)
\]
is $(2(n+1)\rho -b)$-connected, hence and agreement of order $n$.
\end{proof}

This recovers a version \cite[Theorem 4.1]{BE16} as a corollary. 

\begin{cor}\label{BE4.1}
Let $F \in \E_0$. If $F^{(1)}$ is weakly $(\rho,n)$-polynomial and $F(U)$ is $1$-connected whenever $\dim(U) \geq \rho$, then $F$ is weakly $(\rho,n)$-polynomial. 
\end{cor}
\begin{proof}
Apply Theorem \ref{quasi thm} to the fibre sequence 
\[
F^{(1)}(U) \longrightarrow F(U) \longrightarrow F(U \oplus \C). \qedhere
\]
\end{proof}

\begin{ex}\label{example: BU}
The Taylor tower associated to $\BU(-): V \longmapsto \BU(V)$ converges to $\BU(V)$ at $V$ with $\dim(V) \geq 1$. 
\end{ex}
\begin{proof}
By Example \ref{example: sphere functor analytic}, the first derivative of $\BU(-)$ is the sphere functor $\mathbb{S}$ which is weakly polynomial for $\dim(V) \geq 1$ by Example \ref{example: sphere functor analytic}. An application of Corollary \ref{BE4.1} yields the result upon noting that $\BU(V)$ is $1$-connected for $V$ in the given range. 
\end{proof}

We end this section with one of the main results, which shows that the Taylor tower of any representable functor converges. This theorem yields convergence results for many interesting functors. 

\begin{thm}[Theorem H]\label{representable functors are analytic}
Representable functors are weakly polynomial, that is, for all $U \in \J$, the functor $\J(U,-)$ is weakly polynomial. 
\end{thm}
\begin{proof}
As linear isometries are in particular injections, it suffices to prove that the functor
\[
U \longmapsto \J(\C^n , \C^n \oplus U),
\]
is weakly polynomial. The coset quotient projection $U(n-1) \longrightarrow U(n) \longrightarrow S^{2n-1}$ induces a fibre sequence 
\[
\J(\C^{n-1}, \C^{n-1} \oplus U) \longrightarrow \J(\C^n, \C^n \oplus U) \longrightarrow S^{2(\dim(U) +n) -1}.
\]
By induction on $n$, it suffices to prove that $S^{2(\dim(U)+n) -1}$ is weakly $1$-polynomial for all $n\geq 0$, since $T_n$ preserves fibre sequences. This follows for $n\geq 0$ by Example \ref{example: shifted spheres analytic}.
\end{proof}

\begin{ex}
The $n$-sphere functor $n\mathbb{S}: U \longmapsto S^{nU} \cong \J(0,nU)$ is weakly $1$-polynomial.  
\end{ex}

\begin{ex}\label{example: unitary group again}
The functor $\U(-): V \longmapsto \U(V)$ is weakly $1$--polynomial via the identification $\U(V) \cong \J(V,V)$.
\end{ex}

\begin{lem}\label{lemma: properties of analytic functors}
Let $F$ be a weakly $\rho$-polynomial functor. Then, $\Omega F$ is weakly $\rho$-polynomial.
\end{lem}
\begin{proof}
We calculate the connectivity of $\Omega F(U) \longrightarrow T_n \Omega F(U)$ for $U$ with $\dim(U) \geq \rho$, given that 
\[
F(U) \longrightarrow T_nF(U),
\]
is $(2(n+1)\dim(U) -b)$-connected. As homotopy limits commute, it is enough to calculate the connectivity of the map $\Omega F(U) \longrightarrow \Omega T_nF(U)$. This map is $(2(n+1)\dim(U) -b-1)$-connected. 
\end{proof}

\begin{cor}\label{cor: convergence of loops}
If the Taylor tower associated to $F\in \E_0$ converges to $F$, then the Taylor tower associated to $\Omega F$ converges to $\Omega F$.
\end{cor}

We now give an alternative proof that the Taylor tower associated to $\U(-): V \longmapsto \U(V)$ converges to $\U(V)$

\begin{ex}[Example B]\label{example: Unitary group converges}
The Taylor tower associated to $\mathsf{U}(-): V \longmapsto \mathsf{U}(V)$ converges to $\U(-)$ for $\dim(V)\geq 1$.
\end{ex}
\begin{proof}
By Example \ref{example: BU}, the Taylor tower associated to $\mathsf{BU}(-): V \longmapsto \mathsf{BU}(V)$ converges to $\BU(-)$ for $\dim(V) \geq 1$. It follows by Corollary \ref{cor: convergence of loops} that the Taylor tower associated to $\Omega \mathsf{BU}(-) : V \longmapsto \Omega \mathsf{BU}(V)$ converges to $\Omega\BU(-)$ for $\dim(V)\geq 1$. This last functor is levelwise weakly equivalent (in fact homotopy equivalent) to $\mathsf{U}(-): V \longmapsto \mathsf{U}(V)$. 
\end{proof}

\bibliography{references}
\bibliographystyle{alpha}
\end{document}